\documentclass[reqno]{amsart}

\usepackage{color}
\usepackage[pdfstartpage=1,colorlinks=true,bookmarks=false,pdfstartview={FitH}]{hyperref}

\newcommand{\R} {\mathbb R}
\newcommand{\cuad}{{\sqcap\kern-.68em\sqcup}}

\newcommand{\dist}{{\rm dist}\, }

\newcommand{\be}{\begin{equation}}
\newcommand{\ee}{\end{equation}}

\usepackage{amsmath,amssymb}
\usepackage{mathtools}
\usepackage{verbatim}
\usepackage{comment}
\newcommand{\COMMENT}[1]{}
\DeclarePairedDelimiter\abs{\lvert}{\rvert}
\makeatletter
\let\oldabs\abs
\def\abs{\@ifstar{\oldabs}{\oldabs*}}
\newcommand{\eps}{\varepsilon}
\newcommand{\Beta}{\mathrm{B}}
\newcommand{\p}{\partial}
\newcommand{\Div}{\textnormal{div}\,}
\newcommand{\sLap}{(-\Delta)^s}
\newcommand{\sLapx}{(-\partial^2)^s}

\newcommand{\PV}{\textnormal{P.V.}\,}

\newcommand{\norm}[2][]{\left\|{#2}\right\|_{#1}}

\newcommand{\set}[1]{\left\{#1\right\}}

\newtheorem{lem}{Lemma}[section]
\newtheorem{cor}{Corollary}[section]
\newtheorem{thm}{Theorem}[section]

\newtheorem{prop}{Proposition}[section]

\newtheorem{remark}{Remark}[section]
\newcommand{\bremark}{\begin{remark} \em}
\newcommand{\eremark}{\end{remark} }

\numberwithin{equation}{section}

\definecolor{g2}{rgb}{0,0.6,0}
\definecolor{r2}{rgb}{0.8,0,0}

\vbadness=\maxdimen

\begin{document}

\title[Pyramidal traveling wave for fractional Allen-Cahn equations]{ Traveling wave solutions for bistable fractional Allen-Cahn equations with a pyramidal front }

\author{Hardy Chan}
\address{\noindent H. Chan - Department of Mathematics, University of British Columbia, Vancouver, B.C., Canada, V6T 1Z2.}
\email{hardy@math.ubc.ca}

\author{Juncheng Wei}
\address{\noindent J. Wei - Department of Mathematics, University of British Columbia, Vancouver, B.C., Canada, V6T 1Z2.} \email{jcwei@math.ubc.ca}

\thanks{J. Wei is partially supported by NSERC of Canada.}

\begin{abstract}
Using the method of sub-super-solution, we construct a solution of $(-\Delta)^su-cu_z-f(u)=0$ on $\R^3$ of pyramidal shape. Here $(-\Delta)^s$ is the fractional Laplacian of sub-critical order $1/2<s<1$ and $f$ is a bistable nonlinearity. Hence, the existence of a traveling wave solution for the parabolic fractional Allen-Cahn equation with pyramidal front is asserted.

The maximum of planar traveling wave solutions in various directions gives a sub-solution. A super-solution is roughly defined as the one-dimensional profile composed with the signed distance to a rescaled mollified pyramid. In the main estimate we use an expansion of the fractional Laplacian in the Fermi coordinates.
\end{abstract}

\maketitle

\section{Introduction}

\subsection{Traveling waves with local diffusion}
Consider the nonlinear diffusion equations
$$v_t-\Delta{v}-f(v)=0,\qquad\text{in}~\R^n.$$
The study of such equations is initiated by Kolmogorov, Petrovsky and Piskunow \cite{kolmogorov-petrovsky-piskunow} and Fisher \cite{fisher}.
Such reaction-diffusion equations have numerous applications in sciences
(\cite{fisher}, \cite{aronson-weinberger}, \cite{kolmogorov-petrovsky-piskunow}, \cite{allen-cahn}, \cite{fife2}, \cite{bebernes-eberly}, just to name a few)
as a model of genetics and pattern formation in biology, phase transition phenomena in physics, chemical reaction and combustion and many more.

It is natural to look for traveling wave solutions, that is, solutions of the form $v(x,t)=u(x',x_n-ct)$ where $x=(x',x_n)$ and $c$ is the speed. The equation for $u$ reads as
$$-\Delta{u}-cu_{x_n}-f(u)=0,\qquad\text{in}~\R^n.$$
Planar traveling fronts are obtained by further restricting $u(x)=U(x_n)$, resulting in an one-variable ODE
$$-U''-cU'-f(U)=0,\qquad\text{in}~\R.$$

For the KPP nonlinearity $f(t)=t(1-t)$ which is monostable, a planar front exists if $c>2\sqrt{f'(0)}>0$. In the case of cubic bistable nonlinearity $f(t)=-(t-t_0)(t-1)(t+1)$, the nonlinearity determines the speed uniquely by
$$c=\dfrac{\int_{-1}^1f(t)\,dt}{\int_{-1}^1U'(t)^2\,dt}.$$
These classical results are discussed in \cite{fife1}.

The study of non-planar traveling waves with a unbalanced bistable nonlinearity ($t_0\neq0$) is more interesting.
Ninomiya and Taniguchi \cite{ninomiya-taniguchi} proved the existence of a V-shaped traveling wave when $n=2$.
Hamel, Monneau and Roquejoffre \cite{hamel-monneau-roquejoffre2} obtained a higher dimensional analog with cylindrical symmetry.
Taniguchi \cite{taniguchi1} found asymptotically pyramidal waves.
He also constructed traveling waves whose conical front has a level set given by any convex compact set in any dimension $n$ \cite{taniguchi3} (see also \cite{kurokawa-taniguchi}).
Generalized traveling fronts, like curved and pulsating ones, are also considered, notably by Berestycki and Hamel \cite{berestycki-hamel}.

Qualitative properties such as stability and uniqueness of various nonlinearities have also been studied.
The readers are referred to \cite{matano-nara-taniguchi}, \cite{mellet-nolen-roquejoffre-ryzhik}, \cite{shen}, \cite{hamel-roques} and the references therein.

Hereafter we assume that $f\in{C^2}(\R)$ is a more general bistable nonlinearity, that is, there exists $t_0\in(-1,1)$ such that
\begin{equation}\label{bistable}\left\{\begin{array}{l}
f(\pm1)=f(t_0)=0\\
f(t)<0, \quad \forall t\in(-1,t_0)\\
f(t)>0, \quad \forall t\in(t_0,1)\\
f'(\pm1)<0.
\end{array}\right.\end{equation}

\subsection{Fractional Laplacian}
One way to define the fractional Laplacian is via integral operator. Let $0<s<1$ and $n\geq1$ be an integer. Consider the space of functions
$$C_s^{2}(\R^n)=\set{v\in{C}^{2}(\R^n):\int_{\R^n}\dfrac{\abs{v(x)}}{(1+\abs{x})^{n+2s}}\,dx<\infty}.$$
For any function $u\in{C_s^{2}}(\R^n)$, we have the equivalent definitions
\begin{equation*}\begin{split}
\sLap{u}(x)
&=C_{n,s}\PV\int_{\R^n}\!\dfrac{u(x)-u(x+\xi)}{\abs{\xi}^{n+2s}}\,d\xi\\
&=C_{n,s}\PV\int_{\R^n}\!\dfrac{u(x)-u(\xi)}{\abs{x-\xi}^{n+2s}}\,d\xi\\
&=C_{n,s}\int_{\R^n}\!\dfrac{2u(x)-u(x+\xi)-u(x-\xi)}{2\abs{\xi}^{n+2s}}\,d\xi\\
&=C_{n,s}\int_{\R^n}\!\dfrac{u(x)-u(x+\xi)+\chi_D(\xi)\nabla{u}(x)\cdot\xi}{\abs{\xi}^{n+2s}}\,d\xi
\end{split}\end{equation*}
where $D$ is any ball centered at the origin and
$$C_{n,s}=\left(\int_{\R^n}\dfrac{1-\cos(\zeta_1)}{\abs{\zeta}^{n+2s}}\,d\zeta\right)^{-1}=\dfrac{2^{2s}s\Gamma(\frac{n}{2}+s)}{\Gamma(1-s)\pi^{\frac{n}{2}}}.$$

We can also define it as a pseudo-differential operator with symbol $\abs{\xi}^{2s}$, that is, for any $u\in\mathcal{S}(\R^n)$, the Schwartz space of rapidly decaying functions,
$$\widehat{(-\Delta)^{s}u}(\xi)=\abs{\xi}^{2s}\hat{u}(\xi),\qquad\text{for}~\xi\in\R^n.$$
See, for instance, \cite{landkof}.

Caffarelli and Silvestre \cite{caffarelli-silvestre1} considered the localized extension problem
\begin{equation}\label{extension}\begin{cases}
\Div(y^{1-2s}\nabla{v(x,y)})=0,&(x,y)\in\R^n_+:=\R^n\times(0,\infty),\\
v(x,0)=u(x),&x\in\R^n,
\end{cases}\end{equation}
and proved that the fractional Laplacian is some normal derivative
$$(-\Delta)^{s}u(x)=-\dfrac{\Gamma(s)}{2^{1-2s}\Gamma(1-s)}\lim_{y\to0^+}y^{1-2s}v_y(x,y).$$
Hence, the fractional Laplacian is a Dirichlet-to-Neumann map.
The $s$-harmonic extension $v$ of $u$ can be recovered by the convolution $v(\cdot,y)=u\ast{P}_{n,s}(\cdot,y)$ where $P_{n,s}$ is the Poisson kernel
$$P_{n,s}(x,y)=\dfrac{\Gamma(\frac{n}{2}+s)}{\pi^{\frac{n}{2}}\Gamma(s)}\dfrac{y^{2s}}{\left(\abs{x}^2+y^2\right)^{\frac{n}{2}+s}}.$$

The fractional Laplacian can also be understood as the infinitesimal generator of a L\'evy process \cite{bertoin} and it arises in the areas of probability and mathematical finance.

Its mathematical aspects have been studied extensively by many authors, for instance
\cite{silvestre1}, \cite{caffarelli-silvestre1}, \cite{cabre-sire1}, \cite{palatucci-savin-valdinoci}, \cite{gonzalez}, \cite{savin-valdinoci1}, \cite{savin-valdinoci2} and \cite{brandle-colorado-depablo-sanchez}.

In appendix \ref{sect frac lap} we list some useful properties.
When $n=1$ let us also write $\sLap=\sLapx$.

\subsection{The one-dimensional profile}

Consider the equation

\begin{equation}\label{1D}\begin{cases}
\sLapx\Phi(\mu)-k\Phi'(\mu)-f(\Phi(\mu))=0,&\forall\mu\in\R\\
\Phi'(\mu)<0,&\forall\mu\in\R\\
\displaystyle\lim_{\mu\to\pm\infty}\Phi(\mu)=\mp1.
\end{cases}\end{equation}

Gui and Zhao \cite{gui-zhao} proved that

\begin{thm}[Existence of 1-dimensional profile]\label{1Dexist}
For any $s\in(0,1)$ and for any bistable nonlinearity $f\in C^2(\R)$ there exists a unique pair $(k,\Phi)$ such that (\ref{1D}) is satisfied. Moreover, $k>0$ and $\Phi(\mu)$, $\Phi'(\mu)$ decay algebraically as $\abs{\mu}\to\infty$:
$$1-\abs{\Phi(\mu)}=O\left(\abs{\mu}^{-2s}\right) \text{ as } \abs{\mu}\to\infty$$
and
$$0<-\Phi'(\mu)=O\left(\abs{\mu}^{-1-2s}\right) \text{ as } \abs{\mu}\to\infty.$$
\end{thm}

Note that here $\Phi$ is the negative of the profile stated in \cite{gui-zhao}. To fix the phase, we assume that $\Phi(0)=0$.

One may expect that $\Phi''(\mu)$ decays like $\abs{\mu}^{-2-2s}$ as in the almost-explicit example of Cabr\'e and Sire \cite{cabre-sire2} but it would not be as easy to prove because there is no known example of explicit positive function decaying at such rate and satisfying an equation involving the fractional Laplacian.

For our purpose, it is enough to have $\Phi''(\mu)=O\left(\abs{\mu}^{-1-2s}\right)$ as $\abs{\mu}\to\infty$. It is done by a comparison similar to the one in \cite{gui-zhao}. We postpone the proof to appendix \ref{sect 1Ddecay}.

\subsection{Traveling waves with nonlocal diffusion}

Nonlocal reaction-diffusion equations often gives a more accurate model by taking into account long-distance interactions.
Equations involving a convolution with various kernels have been studied, as in
\cite{demasi-gobron-presutti}, \cite{chen}, \cite{bates-chen-chmaj}, \cite{bates-fife-ren-wang}, \cite{wang}, \cite{garroni-muller1}, \cite{garroni-muller2} and \cite{cuitino-koslowski-ortiz}.

From now on let us focus on the case with fractional Laplacian. For $c>k$, we consider a three-dimensional nonlocal diffusion equation
\begin{equation}\label{3D}
\mathcal{L}[u]:=\sLap{u}-cu_z-f(u)=0,\qquad\text{in}~\R^3.
\end{equation}

We say that $v\in{C}(\R^3)\cap{L}^\infty(\R^3)$ is a {\em{sub-solution}} if $v=\max_j{v_j}$ for finitely many $v_j\in{C}^2(\R^3)\cap{L}^\infty(\R^3)$ satisfying $\mathcal{L}[v_j]\leq0$ for each $j$.
A {\em{super-solution}} is defined similarly.

In order to state our main result, let us define a pyramid in the sense of \cite{taniguchi1}. Let $m_*=\sqrt{c^2-k^2}/k>0$ and let $N\geq3$ be an integer. Let
$$\set{(a_j,b_j)\in\R^2\mid1\leq{j}\leq{N}}$$
be pairs of real numbers satisfying the following properties.
\begin{itemize}
\item $a_j^2+b_j^2=m_*^2$, for each $1\leq{j}\leq{N}$;
\item $a_{j}b_{j+1}-a_{j+1}b_{j}>0$, for each $1\leq{j}\leq{N}$, where we have set $a_{n+1}=a_1$ and $b_{n+1}=b_1$;
\item $(a_{j},b_{j})\neq(a_{j'},b_{j'})$ if $j\neq{j'}$.
\end{itemize}
For each $1\leq{j}\leq{N}$ we define the function $h_j(x,y)=a_{j}x+b_{j}y$ and we define
$$h(x,y)=\displaystyle\max_{1\leq{j}\leq{N}}h_j(x,y).$$
Let us call $\set{(x,y,z)\in\R^3\mid{z}=h(x,y)}$ a pyramid. We decompose $\R^2=\bigcup_{j=1}^{N}\Omega_j$, where
$$\Omega_{j}=\set{(x,y)\in\R^2\mid{h}(x,y)=h_j(x,y)}.$$
It is clear that $h_j(x,y)\geq0$ for $(x,y)\in\overline{\Omega_j}$ and hence $h(x,y)\geq0$ for all $(x,y)\in\R^2$. By the assumptions on $(a_j,b_j)$, we see that $\Omega_j$ are oriented counter-clockwise. The set of all edges of the pyramid is $\Gamma=\bigcup_{j=1}^N\Gamma^j$ where
$$\Gamma^j=\set{(x,y,z)\in\R^3\mid{z}=h_j(x,y)~\text{for}~(x,y)\in\p\Omega_j}.$$
Denote
$$\Gamma_R=\set{(x,y,z)\in\R^3\mid\dist((x,y,z),\Gamma)>R}.$$
The projection of $\Gamma$ on $\R^2\times\set{0}$ is identified as $E=\bigcup_{j=1}^{N}\p\Omega_j\subset\R^2$.

In Section \ref{sect subsol} we show that
\begin{equation}\label{subsol def}
v(x,y,z)=\Phi\left(\dfrac{k}{c}(z-h(x,y))\right).
\end{equation}
is a sub-solution of \eqref{3D}. In Sections \ref{sect mol pyr} and \ref{sect supersol}, we obtain a super-solution in the form
\begin{equation}\label{supersol def}
V(x,y,z)=\Phi\left(\dfrac{z-\alpha^{-1}\varphi(\alpha{x},\alpha{y})}{\sqrt{1+\abs{\nabla\varphi(\alpha{x},\alpha{y})}^2}}\right)+{\eps}S(\alpha{x},\alpha{y}).
\end{equation}
For the precise definition, see \eqref{def varphi}, \eqref{def S}, \eqref{param eps} and \eqref{param alpha}.

The main result is as follows.

\begin{thm}\label{main}
Given any bistable nonlinearity $f$ and any $c>k$, where $k>0$ is given in Theorem \ref{1Dexist}, there exists a solution $u$ of \eqref{3D} such that $v<u<V$ in $\R^3$ where $v$ and $V$ are defined by \eqref{subsol def} and \eqref{supersol def} respectively. In particular,
$$\lim_{R\to\infty}\sup_{(x,y,z)\in\Gamma_R}\,\abs{u(x,y,z)-v(x,y,z)}=0.$$
\end{thm}

\medskip

\section{Motivation}
\label{sect motiv}

It is worthwhile to sketch the idea in \cite{taniguchi1} for the standard Laplacian case $s=1$.

Suppose there exists a one-dimensional solution $\Phi(\mu)$ of
$$-\Phi''(\mu)-k\Phi'(\mu)-f(\Phi(\mu))=0.$$
Let $\rho:\R^2\to(0,1]$ be a smooth radial mollifier satisfying $\int\rho=1$ and decaying exponentially at infinity, and $h(x,y)=(\sqrt{2}k)^{-1}\sqrt{c^2-k^2}(\abs{x}+\abs{y})$ be a square pyramid.
Let $\varphi=\rho\ast{h}$ be its mollification and $S(x,y)=c\left/\sqrt{1+\abs{\nabla\varphi(x,y)}^2}\right.-k$ be an auxiliary function.

It is easy to check that $v(x,y,z)=\Phi\left((k/c)(z-h(x,y))\right)$, as a maximum of solutions, is a sub-solution of
$$-\Delta{u}-cu_z-f(u)=0.$$
For $\alpha,\eps\in(0,1)$, define
$$V(x,y,z)=\Phi\left(\hat\mu\right)+\eps{S}(\alpha{x},\alpha{y}),$$
where
$$\hat\mu=\dfrac{z-\alpha^{-1}\varphi(\alpha{x},\alpha{y})}{\sqrt{1+\abs{\nabla\varphi(\alpha{x},\alpha{y})}^2}}.$$
Introducing the rescaled function $\Phi_\alpha(\mu)=\Phi(\alpha^{-1}\mu)$, we can also write
$$V(x,y,z)=\Phi_\alpha\left(\bar\mu(\alpha{x},\alpha{y},\alpha{z})\right)+\eps{S}(\alpha{x},\alpha{y}),$$
where
$$\bar\mu(x,y,z)=\dfrac{z-\varphi(x,y)}{\sqrt{1+\abs{\nabla\varphi(x,y)}^2}}.$$
$V$ will be a super-solution if $\alpha$ and $\eps$ are small. Indeed,
\begin{equation*}\begin{split}
-\Delta V
&=-\Phi''(\hat\mu)+\alpha{R}\\
&=k\Phi'(\hat\mu)+f(\Phi(\hat\mu))+\alpha{R}\\
-cV_z
&=-\dfrac{c}{\sqrt{1+\abs{\nabla\varphi(\alpha{x},\alpha{y})}^2}}\Phi'(\hat\mu)
\end{split}\end{equation*}
where $R=R(\alpha{x},\alpha{y};\hat\mu,\alpha,\eps)$ is bounded and each of its terms contains a second or third order derivative of $\varphi$. Then we have
\begin{equation*}\begin{split}
&\quad\;\mathcal{L}[V]\\
&=-\Delta{V}-cV_z-f(V)\\
&=S(\alpha{x},\alpha{y})\left(-\Phi'(\hat\mu)-\eps\int_0^1\!f'(\Phi(\hat\mu)+t\eps{S})\,dt+\alpha\dfrac{R(\alpha{x},\alpha{y})}{S(\alpha{x},\alpha{y})}\right)
\end{split}\end{equation*}
As $R(\alpha{x},\alpha{y})$ decays at an (exponential) rate not lower than $S(\alpha{x},\alpha{y})$ as $\abs{x},\abs{y}\to\infty$, the last term is bounded. By choosing $\alpha\ll\eps\ll1$ small, we are left with the main term $-\Phi'(\hat\mu)$ or $-\int_0^1f'\,dt$, depending on the magnitude of $\hat\mu$, which is positive.

For $1/2<s<1$, we cannot compute the Laplacian pointwisely using the chain rule but we can still arrange $\mathcal{L}[V]$, in terms of the difference $\sLap(\Phi(\hat\mu))-(\sLapx\Phi)(\hat\mu)$, as
\begin{equation*}\begin{split}
\sLap{V}=k\Phi'(\hat\mu)+f(\Phi(\hat\mu))+\tilde{R}
\end{split}\end{equation*}
where the ``remainder''
\begin{equation*}\begin{split}
\tilde{R}
&=\tilde{R}(\alpha{x},\alpha{y};\hat\mu,\alpha,\eps)\\
&=\sLap(\Phi(\hat\mu))-(\sLapx\Phi)(\hat\mu)+\eps\sLap(S(\alpha{x},\alpha{y}))
\end{split}\end{equation*}
is now a non-local term. We still have
$${\mathcal{L}}[V]=S(\alpha{x},\alpha{y})\left(-\Phi'(\hat\mu)-\eps\int_0^1\!f'(\Phi(\hat\mu)+t\eps{S})\,dt+\dfrac{\tilde{R}(\alpha{x},\alpha{y})}{S(\alpha{x},\alpha{y})}\right).$$

In terms of the rescaled one-dimensional solution $\Phi_\alpha$, by the homogeneity of the fractional Laplacian (see Lemma \ref{homo}), we have $\tilde{R}=\alpha^{2s}(R_1+R_2)$ where
\begin{equation*}\begin{split}
R_1
&=R_1(x,y,z;\alpha)\\
&=\sLap(\Phi_\alpha(\bar\mu(\alpha{x},\alpha{y},\alpha{z})))-(\sLapx\Phi_\alpha)(\bar\mu(\alpha{x},\alpha{y},\alpha{z}))\\
R_2
&=R_2(x,y;\alpha,\eps)\\
&=\eps(\sLap{S})(\alpha{x},\alpha{y}).
\end{split}\end{equation*}
It remains to show that $R_1,R_2=o\left(\alpha^{-2s}\right)$ as $\alpha\to0$, uniformly in $(x,y,z)\in\R^3$.
This will be done in sections \ref{sect mol pyr} and \ref{sect supersol}.

\begin{remark}[On the sub-criticality of $s$]
Since $\sLapx{S}$ cannot decay any faster than $\abs{(x,y)}^{-2s}$ as $\dist((x,y),E)\to\infty$ by its non-local nature, this argument will work out only if $S$ has an algebraic decay.
Hence, instead of an exponentially small mollifier, we must choose $\rho(x,y)=\Omega\left(\abs{(x,y)}^{-1-2s}\right)$.
On the other hand, in order that $\varphi$ is well defined, it is necessary to take $\rho(x,y)=O\left(\abs{(x,y)}^{-2}\right)$.
This forces $-1-2s<-2$, or $s>1/2$.
\end{remark}

\medskip
\section{The sub-solution}
\label{sect subsol}

We show that $v$ given by \eqref{subsol def} is a sub-solution.

\begin{prop}\label{subsol}
$v$ is a sub-solution of equation \eqref{3D}.
\end{prop}

\begin{proof}
Let us define, for each $j=1,\dots,N$,
\begin{equation}\label{vj}
v_j(x,y,z)=\Phi\left(\dfrac{k}{c}(z-h_j(x,y))\right)=\Phi\left(\dfrac{k}{c}(z-a_j{x}-b_j{y})\right).
\end{equation}
Since $\Phi'<0$, we see that $v=\displaystyle\max_{1\leq{j}\leq{N}}v_j$.

By Lemma \ref{chain},
\begin{equation*}\begin{split}
\sLap{v_j}(x,y,z)
&=\left(\left(\dfrac{k}{c}\right)^2(1+a_j^2+b_j^2)\right)^s\sLapx\Phi\left(\dfrac{k}{c}(z-a_j{x}-b_j{y})\right)\\
&=\sLapx\Phi\left(\dfrac{k}{c}(z-a_{j}x-b_{j}y)\right).
\end{split}\end{equation*}
By \eqref{1D}, we have
$$\mathcal{L}[v_j]={\mathcal{L}}\left(\Phi\left(\dfrac{k}{c}(z-a_{j}x-b_{j}y)\right)\right)=0.$$
By definition, $v$ is a sub-solution of \eqref{3D}.
\end{proof}

\medskip
\section{The mollified pyramid and an auxiliary function}
\label{sect mol pyr}
Most of the materials in this section is technical and is a variation of those in \cite{taniguchi1}.

We define a radial mollifier $\rho\in{C}^{\infty}(\R^3)$ by $\rho(x,y)=\tilde\rho\left(\sqrt{x^2+y^2}\right)$, where $\tilde\rho\in C^\infty([0,\infty))$ satisfies the following properties:
\begin{itemize}
\item $0<\tilde\rho(r)\leq1$ and $\tilde\rho'(r)\leq0$ for $r>0$,
\item $\displaystyle2\pi\int_{0}^\infty\!r\tilde\rho(r)\,dr=1$,
\item $\tilde\rho(r)=1$ for $0\leq{r}\leq\tilde{r_0}\ll1$,
\item $\tilde\rho(r)=\tilde\rho_{0}r^{-{2s}-2}$ for $r\geq r_0\gg2$, where $\tilde\rho_{0}>0$ is chosen such that
    \begin{equation}\label{rhoconst}
    \dfrac{\tilde\rho_0}{{2s}}\Beta\left(\dfrac12,\dfrac12+s\right)=1
    \end{equation}
    and $r_0$ satisfies
    \begin{equation}\label{r0}
    {2s}(m_*^2+2)(2r_0)^{-{2s}}<1.
    \end{equation}
\end{itemize}
Define
\begin{equation}\label{def varphi}
\varphi(x,y)=\rho\ast{h}(x,y).
\end{equation}
We call $z=\varphi(x,y)$ a mollified pyramid. Define also an auxiliary function
\begin{equation}\label{def S}\begin{split}
S(x,y)
&=\dfrac{c}{\sqrt{1+\abs{\nabla\varphi(x,y)}^2}}-k\\
&=\dfrac{m_*^2-\abs{\nabla\varphi}^2}{\sqrt{1+\abs{\nabla\varphi(x,y)}^2}\left(c+k\sqrt{1+\abs{\nabla\varphi(x,y)}^2}\right)}.
\end{split}\end{equation}

By direct computation, we have

\begin{lem}\label{varphi1}
For any integers $i_1\geq0$ and $i_2\geq0$, with $i_1+i_2\leq3$,
$$\sup_{(x,y)\in\R^2}\abs{\p_x^{i_1}\p_y^{i_2}\varphi(x,y)}<\infty.$$
For all $(x,y)\in\R^2$, we have
$$h(x,y)<\varphi(x)\leq{h}(x,y)+2\pi{m_*}\int_0^\infty\!r^2\tilde\psi(r)\,dr$$ and $\abs{\nabla\varphi(x,y)}<m_*$. Hence, $0<S(x,y)\leq{c-k}$.
\end{lem}

\begin{proof}
The proof can be found in \cite{taniguchi1}, with a slight variation that there is a constant $C_{\tilde\rho}$ depending only on $\tilde\rho$ such that
$$\abs{\p_x^{i_1}\p_y^{i_2}\rho(x,y)}\leq{C_{\tilde\rho}}(x^2+y^2)^{-\frac{i_1+i_2}{2}}\rho(x,y),\quad\text{for}~x^2+y^2\geq{r_0^2}$$
and
$$\abs{\p_x^{i_1}\p_y^{i_2}\rho(x,y)}\leq{C_{\tilde\rho}}\quad\text{for~all}~(x,y)\in\R^2.$$

\COMMENT{}
\end{proof}

In the rest of this section, we study the behavior of $\varphi(x,y)-h(x,y)$ and $S(x,y)$ as well as their derivatives. It turns out that both of them depend on the distance from the edge of the pyramid.

The behavior of these functions of interest can be expressed using a ``mollified negative part''. Write $x_+=\max\set{x,0}$ and $x_-=\max\set{-x,0}$. Define $P:[0,\infty)\to(0,\infty)$ by
\begin{equation*}\begin{split}
P(x)
&=\int_{\R^2}\!\rho(x',y')(x-x')_{-}\,dx'dy'\\
&=-\int_{\R^2}\!\rho(x',y')(x'-x)_{+}\,dx'dy'\\
&=-\int_x^\infty\!\left(\int_{-\infty}^\infty\!\rho(x',y')\,dy'\right)(x-x')\,dx'.
\end{split}\end{equation*}

Let us state the properties of $P$.
\begin{lem}\label{P}
$P$ is in the class ${C}^3([0,\infty))$. Let $0\leq{i}\leq3$. There exists a constant $C_P=C_P(s,\tilde\rho)$ such that $\norm[C^3([0,\infty))]{P}\leq{C_P}$ and
$$C_P^{-1}(1+x)^{-{2s}+1-i}\leq\abs{P^{(i)}(x)}\leq{C_P}(1+x)^{-{2s}+1-i}.$$
Moreover, for $x>0$ we have $(-1)^{i}P^{(i)}(x)>0$ and if $x>{r_0}$, then
\begin{equation*}\begin{split}
P(x)
&=\dfrac{1}{({2s}-1)x^{{2s}-1}}
\COMMENT{}
.
\end{split}\end{equation*}
In particular, $P^{(i)}(x)\to0$ as $x\to\infty$.
\end{lem}

\begin{proof}
Clearly, $P\in{C}^3([0,\infty))$ and if $x>0$, then
\begin{equation*}\begin{split}
P(x)
&=-\int_x^\infty\!\left(\int_{-\infty}^\infty\!\rho(x',y)\,dy\right)(x-x')\,dx'>0\\
P'(x)
&=-\int_x^\infty\!\left(\int_{-\infty}^\infty\!\rho(x',y)\,dy\right)\,dx'<0\\
P''(x)
&=\int_{-\infty}^\infty\!\rho(x,y)\,dy>0\\
P^{(3)}(x)
&=\int_{-\infty}^\infty\!\dfrac{x}{\sqrt{x^2+y^2}}\tilde\rho'\left(\sqrt{x^2+y^2}\right)\,dy<0.
\end{split}\end{equation*}
For $x\geq{r_0}$, we have by \eqref{rhoconst},
\begin{equation*}\begin{split}
P(x)
&=-\int_x^\infty\!\left(\int_{-\infty}^\infty\!\dfrac{\tilde\rho_0}{((x')^2+y^2)^{1+s}}\,dy\right)(x-x')\,dx'\\
&=\tilde\rho_0\Beta\left(\dfrac12,\dfrac{1+{2s}}{2}\right)\int_x^\infty\!\dfrac{x'-x}{(x')^{1+{2s}}}\,dx'\\
&={2s}\left(\dfrac{1}{({2s}-1)x^{{2s}-1}}-\dfrac{x}{{2s}{x}^{{2s}}}\right)\\
&=\dfrac{1}{({2s}-1)x^{{2s}-1}}.
\end{split}\end{equation*}
The decay of the derivatives follows. Since they all have a sign, we have for any $x>0$,
\begin{equation*}\begin{split}
0<P(x)
&<P(0)
=\int_0^\infty\int_{-\infty}^\infty\!x\rho(x,y)\,dydx
=2\int_0^\infty\!r^2\tilde\rho(r)\,dr\\
0<-P'(x)
&<-P'(0)
=\int_0^\infty\int_{-\infty}^\infty\!\rho(x,y)\,dydx
=\dfrac12\\
0<P''(x)
&<P''(0)
=\int_{-\infty}^\infty\!\rho(0,y)\,dy
=2\int_0^\infty\!\tilde\rho(r)\,dr
\end{split}\end{equation*}
To get a bound for $P^{(3)}(x)$, we consider two cases. If $x>r_0$, then
$$\abs{P^{(3)}(x)}<-P^{(3)}(r_0)=\dfrac{{2s}({2s}+1)}{r_0^{{2s}+2}}.$$
If $x\leq{r_0}$, then we use the change of variable $r=\sqrt{x^2+y^2}$, to obtain
\begin{equation*}\begin{split}
\abs{P^{(3)}(x)}
&=2\int_0^\infty\!\dfrac{x}{\sqrt{x^2+y^2}}\tilde\rho'\left(\sqrt{x^2+y^2}\right)\,dy\\
&=2\int_x^\infty\!\dfrac{x}{\sqrt{r^2-x^2}}\abs{\tilde\rho'(r)}\,dr\\
&\leq\sqrt{2x}\int_x^\infty\!\dfrac{\abs{\tilde\rho'(r)}}{\sqrt{r-x}}\,dr\\
&\leq\sqrt{2r_0}\left(\int_0^{r_0}\!\dfrac{\abs{\tilde\rho(r+x)}}{\sqrt{r}}\,dr+\int_{r_0}^\infty\!\dfrac{\tilde\rho_0({2s}+2)}{\sqrt{r}(r+x)^{3+{2s}}}\,dr\right)\\
&\leq\sqrt{2}\left(2r_0\norm[{L^\infty([0,2r_0])}]{\tilde\rho'}+\tilde\rho_0r_0^{-2-{2s}}\right).
\end{split}\end{equation*}
This finishes the proof.
\end{proof}

To estimate $\varphi(x,y)-h(x,y)$, it suffices to fix $(x,y)\in\overline{\Omega_j}$ and study $\tilde\varphi_j(x,y)=\varphi(x,y)-h_j(x,y)=\rho\ast(h-h_j)(x,y)$. By this definition, we expect that $\tilde\varphi_j(x,y)$ to be controlled by the distance from $(x,y)$ to the boundary $\p\Omega_j$ because this distance determine the size of a neighborhood of $(x,y)$ such that $h=h_j$.

To fix the notation, we observe that $\overline{\Omega_j}\cap\overline{\Omega_{j\pm1}}$ is a half line on which $h_j(x,y)-h_{j\pm1}(x,y)=(a_j-a_{j\pm1})x+(b_j-b_{j\pm1})y=0$. Let us write, for $(x,y)\in\overline{\Omega_j}$,
\begin{equation*}\begin{split}
m_j^\pm
&=\sqrt{(a_j-a_{j\pm1})^2+(b_j-b_{j\pm1})^2}\\
\lambda_j^\pm
=\lambda_j^\pm(x,y)
&=\dist((x,y),\Omega_{j\pm1})
=\dfrac{(a_j-a_{j\pm1})x+(b_j-b_{j\pm1})y}{m_j^\pm}\\
\lambda_j
=\lambda_j(x,y)
&=\dist((x,y),\p\Omega_j)
=\min\set{\lambda_j^+,\lambda_j^-},\\
\end{split}\end{equation*}
where $\Omega_{j\pm1}$ is understood to be $\Omega_{j\pm1\pmod{N}}$. We also write
$$\lambda=\lambda(x,y)=\dist((x,y),E)=\min_{1\leq{j}\leq{N}}\lambda_j.$$

By the above motivation, we express $\tilde\varphi_j$ as
\begin{equation*}\begin{split}
\tilde\varphi_j(x,y)
&=\rho\ast(h_j-h_{j+1})_{-}(x,y)+\rho\ast(h_j-h_{j-1})_{-}(x,y)+\rho\ast{g_j}(x,y)\\
&=m_j^{+}P(\lambda_j^+)+m_j^{-}P(\lambda_j^-)+\rho\ast{g_j}(x,y)
\end{split}\end{equation*}
where $g_j=h-h_j-(h_j-h_{j+1})_{-}-(h_j-h_{j-1})_{-}$ vanishes on $\overline{\Omega_{j-1}}\cup\overline{\Omega_j}\cup\overline{\Omega_{j+1}}$.

We prove that $\rho\ast{g_j}$ is an error term, up to three derivatives.
\begin{lem}\label{varphi2}
There exists constants $C_g=C_g(\tilde\rho)$ and $\gamma>1$ such that for any $1\leq{j}\leq{N}$, any $(x,y)\in\overline{\Omega_j}$, and any integers $i_1\geq0$, $i_2\geq0$ with $i_1+i_2\leq3$, we have
$$\abs{\p_x^{i_1}\p_y^{i_2}(\rho\ast{g_j})(x,y)}\leq{C_g}\left(1+\sqrt{x^2+y^2}\right)^{-i_1-i_2}P(\gamma\lambda_j)$$
and, in particular,
$$\abs{\p_x^{i_1}\p_y^{i_2}(\rho\ast{g_j})(x,y)}\leq{C_g}\abs{P^{(i_1+i_2)}(\gamma\lambda_j)}.$$
\end{lem}

\begin{proof}
We first claim that $\abs{g_j}\leq3\abs{h-h_j}$ on the whole $\R^2$. Indeed, by the definition of $g_j$, we have $g_j\leq{h-h_j}$ and $g_j\geq0$ on $(\set{h_j\leq{h_{j+1}}}\cap\set{h_j\leq{h_{j+1}}})^c$. On $\set{h_j\leq{h_{j+1}}}\cap\set{h_j\leq{h_{j+1}}}$, $g_j=h+h_j-h_{j+1}-h_{j-1}$ and so $g_j+(h-h_j)=2h-h_{j+1}-h_{j-1}$. Thus, $0\leq{g_j}+(h-h_j)\leq2(h-h_j)$ on $\R^2$. Therefore,
$$\abs{g_j}\leq\abs{g_j+(h-h_j)}+\abs{h-h_j}=3(h-h_j),$$
as we claimed.

We have
\begin{equation*}\begin{split}
\abs{\p_x^{i_1}\p_y^{i_2}(\rho\ast{g_j})}
&=\abs{(\p_x^{i_1}\p_y^{i_2}\rho)\ast{g_j}}\\
&=\abs{\sum_{k\neq{j},j\pm1}\int_{\Omega_k}\!(\p_x^{i_1}\p_y^{i_2}\rho)(x-x',y-y')g_j(x',y')\,dx'dy'},\\
\end{split}\end{equation*}
where $k\neq{j\pm1}$ is understood as $k\not\equiv{j\pm1}\pmod{N}$.

Suppose $x^2+y^2>r_0^2$. From the proof of Lemma \ref{varphi1}, we can find a constant ${C_{\tilde\rho}}$ such that
$$\abs{\p_x^{i_1}\p_y^{i_2}\rho(x-x',y-y')}\leq{C_{\tilde\rho}}((x-x')^2+(y-y')^2)^{-\frac{i_1+i_2}{2}}\rho(x-x',y-y').$$
For $(x',y')\notin\Omega_j$, $(x-x')^2+(y-y')^2\geq{x}^2+y^2$. Hence,
\begin{equation*}\begin{split}
&\quad\,\,\abs{\p_x^{i_1}\p_y^{i_2}(\rho\ast{g_j})}\\
&\leq3{C_{\tilde\rho}}(x^2+y^2)^{-\frac{i_1+i_2}{2}}\sum_{k\neq{j},j\pm1}\int_{\Omega_k}\!\rho(x-x',y-y')(h_k-h_j)(x',y')\,dx'dy'\\
&=3{C_{\tilde\rho}}(x^2+y^2)^{-\frac{i_1+i_2}{2}}\sum_{k\neq{j},j\pm1}\int_{(x,y)-\Omega_k}\!\rho(x',y')(h_k-h_j)(x-x',y-y')\,dx'dy'\\
&\leq3{C_{\tilde\rho}}(x^2+y^2)^{-\frac{i_1+i_2}{2}}\sum_{k\neq{j},j\pm1}\int_{(x,y)-\Omega_k}\!\rho(x',y')\Big((h_j-h_k)(x',y')\\
&\hspace{8.5cm}-(h_j-h_k)(x,y)\Big)\,dx'dy'\\
\end{split}\end{equation*}

By a simple use of intermediate value theorem, we have
$$\set{(x,y)\in\R^2\mid(h_j-h_k)(x,y)=0}\subset\left(\R^2\backslash\overline{\Omega_j}\right)\cup\set{(0,0)},$$
that is, roughly, the line $\set{h_j=h_k}$ is contained in exterior of $\Omega_j$. Therefore, we can find a constant $\gamma>1$, depending only on the configuration of the $\Omega_j$'s, such that $\dist((x,y),\set{h_j=h_k})>\gamma\lambda_j$. But for $(x,y)\in\Omega_j$,
$$\dist((x,y),\set{h_j=h_k})=\dfrac{(h_j-h_k)(x,y)}{m_{jk}}>0,$$
where $m_{jk}=\sqrt{(a_j-a_k)^2+(b_j-b_k)^2}$. Hence,
$$\dfrac{(h_j-h_k)(x,y)}{m_{jk}}>\gamma\lambda_j.$$
Note also that by a rotation, we can write
$$P(x)=\int_{\R^2}\!\rho(x',y')\left(\dfrac{(h_k-h_j)(x',y')}{m_{jk}}-x\right)_{+}\,dx'dy'.$$

We now have
\begin{equation*}\begin{split}
&\quad\,\,\abs{\p_x^{i_1}\p_y^{i_2}(\rho\ast{g_j})}\\
&\leq3{C_{\tilde\rho}}\left(\max_{1\leq{j,k}\leq{N}}m_{jk}\right)(x^2+y^2)^{-\frac{i_1+i_2}{2}}\\
&\qquad\cdot\sum_{k\neq{j},j\pm1}\int_{(x,y)-\Omega_k}\!\rho(x',y')\Bigg(\dfrac{(h_j-h_k)(x',y')}{m_{jk}}-\dfrac{(h_j-h_k)(x,y)}{m_{jk}}\Bigg)\,dx'dy'\\
&\leq3{C_{\tilde\rho}}\left(\max_{1\leq{j,k}\leq{N}}m_{jk}\right)(x^2+y^2)^{-\frac{i_1+i_2}{2}}\\
&\qquad\cdot\sum_{k\neq{j},j\pm1}\int_{\R^2}\!\rho(x',y')\Bigg(\dfrac{(h_j-h_k)(x',y')}{m_{jk}}-\dfrac{(h_j-h_k)(x,y)}{m_{jk}}\Bigg)_+\,dx'dy'\\
&\leq3{C_{\tilde\rho}}\left(\max_{1\leq{j,k}\leq{N}}m_{jk}\right)(x^2+y^2)^{-\frac{i_1+i_2}{2}}\sum_{k\neq{j},j\pm1}P\left(\dfrac{(h_j-h_k)(x,y)}{m_{jk}}\right)\\
&\leq3{C_{\tilde\rho}}\left(\max_{1\leq{j,k}\leq{N}}m_{jk}\right)(x^2+y^2)^{-\frac{i_1+i_2}{2}}P(\gamma\lambda_j)
\end{split}\end{equation*}

If $x^2+y^2\leq{r_0^2}$, then we use the fact
$$\abs{\p_x^{i_1}\p_y^{i_2}\rho(x-x',y-y')}\leq{C_{\tilde\rho}}\rho(x-x',y-y'),$$
which is also mentioned in Lemma \ref{varphi1}. By the exact same argument, the first estimate is established.

To obtain the second one, noting that since the angle of $\Omega_j$ (that is, the angle between the lines $h_j=h_{j+1}$ and $h_j=h_{j-1}$) is less than $\pi$, we have $\sqrt{x^2+y^2}\geq\lambda_j$. Now the right hand side of the two estimates are bounded for $\lambda_j\leq{r_0}$ and is $O\left(\lambda_j^{-{2s}+1-i_1-i_2}\right)$ for $\lambda_j\geq{r_0}$. Hence, by enlarging $C_g$ if necessary, the proof is completed.
\end{proof}

\begin{lem}\label{varphi3}
There exists a constant $C_\varphi=C_\varphi(\tilde\rho,\gamma)$ such that for any non-negative integers $i_1$, $i_2$ with $0\leq{i_1+i_2}\leq3$ and any $(x,y)\in\R^2\backslash{E}$, we have
$$\abs{\p_x^{i_1}\p_y^{i_2}(\varphi(x,y)-h(x,y))}\leq{C_\varphi}\abs{P^{(i_1+i_2)}(\lambda)},$$
and for any $(x,y)\in\R^2$,
$$\varphi(x,y)-h(x,y)\geq{C_\varphi^{-1}}P(\lambda).$$
In particular, if $2\leq{i_1+i_2}\leq3$, then for any $(x,y)\in\R^2$, there holds
$$\abs{\p_x^{i_1}\p_y^{i_2}\varphi(x,y)}\leq{C_\varphi}\abs{P^{(i_1+i_2)}(\lambda)}.$$
\end{lem}

\begin{proof}
Without loss of generality, let $(x,y)\in\Omega_j$. We have
$$\varphi(x,y)-h_j(x,y)=m_j^{+}P(\lambda_j^{+})+m_j^{-}P(\lambda_j^{-})+\rho\ast{g_j}(x,y)$$
and so
\begin{equation*}\begin{split}
\p_x^{i_1}\p_y^{i_2}\tilde\varphi_j(x,y)
&=m_j^{+}(\p_x\lambda_j^{+})^{i_1}(\p_y\lambda_j^{+})^{i_2}P^{(i_1+i_2)}(\lambda_j^{+})\\
&\qquad+m_j^{-}(\p_x\lambda_j^{-})^{i_1}(\p_y\lambda_j^{-})^{i_2}P^{(i_1+i_2)}(\lambda_j^{-})+\p_x^{i_1}\p_y^{i_2}(\rho\ast{g_j})(x,y)\\
&=(m_j^{+})^{1-i_1-i_2}(a_j-a_{j+1})^{i_1}(b_j-b_{j+1})^{i_2}P^{(i_1+i_2)}(\lambda_j^{+})\\
&\qquad+(m_j^{-})^{1-i_1-i_2}(a_j-a_{j-1})^{i_1}(b_j-b_{j-1})^{i_2}P^{(i_1+i_2)}(\lambda_j^{-})\\
&\qquad+\p_x^{i_1}\p_y^{i_2}(\rho\ast{g_j})(x,y).
\end{split}\end{equation*}
By Lemma \ref{varphi2}, we have
\begin{equation*}\begin{split}
\abs{\p_x^{i_1}\p_y^{i_2}\varphi(x,y)}
&\leq\tilde{C}_1\abs{P^{(i_1+i_2)}(\lambda_j)}+C_{g}\abs{P^{(i_1+i_2)}(\gamma\lambda_j)}\\
&\leq(\tilde{C}_1+C_{g})\abs{P^{(i_1+i_2)}(\lambda_j)}.
\end{split}\end{equation*}

To get the lower bound, we estimate $\tilde\varphi_j$ in three regions as follows. We choose $r_1>r_0$ such that for all $\lambda_j\geq{r_1}$,
$${C_g}P(\gamma\lambda_j)<\dfrac{1}{2}\min_{1\leq{j}\leq{N}}\left(\min\set{m_j^+,m_j^-}\right)P(\lambda_j).$$
On $\set{(x,y)\in\overline{\Omega_j}\mid\lambda_j\geq{r_1}}$,
$$\tilde\varphi_j(x,y)\geq\dfrac{1}{2}\min_{1\leq{j}\leq{N}}\left(\min\set{m_j^+,m_j^-}\right)P(\lambda_j).$$
It suffices to show that
$$\inf\set{\varphi_j(x,y)\,\Big|\,\lambda_j<r_1,\,\sqrt{x^2+y^2}\geq{r_2}}>0$$
for some $r_2>0$, since $\tilde\varphi_j$ attains a positive minimum in any compact set. Without loss of generality we assume that $\lambda_j^+<r_1$. Let $r_2$ be large enough such that $B_{r_1+2}\subset\overline{\Omega_j}\cup\overline{\Omega_{j+1}}$. Note that the outward unit normal of $\Omega_j$ on $\p\Omega_j\cap\p\Omega_{j+1}$ is $-\dfrac{{\bf{a}}_j-{\bf{a}}_{j+1}}{m_j^+}$. We estimate
\begin{equation*}\begin{split}
\tilde\varphi_j(x,y)
&=\rho\ast(h-h_j)(x,y)\\
&=\int_{\Omega_j^c}\!\tilde\rho(x-x',y-y')(h-h_j)(x',y')\,dx'dy'\\
&\geq\tilde\rho(r_1+2)\int_{B_{r_1+2}(x,y)\cap\Omega_{j+1}}(h_{j+1}-h_j)(x',y')\,dx'dy'\\
&\geq\tilde\rho(r_1+2)\int_{B_*}(h_{j+1}-h_j)(x',y')\,dx'dy'\\
\end{split}\end{equation*}
where $B_*\subset{B_{r_1+2}(x,y)\cap\Omega_{j+1}}$ is the half ball
$$B_*=B_{1}\left((x,y)-(\lambda_j^{+}+1)\dfrac{{\bf{a}}_j-{\bf{a}}_{j+1}}{m_j^+}\right)\cap\set{\dfrac{h_j-h_{j+1}}{m_j^+}\geq1}.$$
Now
\begin{equation*}\begin{split}
\tilde\varphi_j(x,y)
&\geq\tilde\rho(r_1+2)\int_{B_*}{m_j^+}\,dx'dy'\\
&=\dfrac{\pi{m_j^+}}{2}\tilde\rho(r_1+2).
\end{split}\end{equation*}
This completes the proof.
\end{proof}

Now, in terms of $P$ and $\lambda$ we state a key lemma concerning the decay properties $S(x,y)$ and its derivatives. In the following, it is convenient to use the vector notation ${\bf{a}}_j=(a_j,b_j)$.

\begin{lem}\label{S}
There exists a constant $C_S=C_S(s,\tilde\rho,\gamma)$ such that for any $(x,y)\in\R^2$ and any non-negative integers $i_1$, $i_2$ with $1\leq{i_1+i_2}\leq2$,
$$C_S^{-1}\abs{P'(\lambda)}\leq{S}(x,y)\leq{C_S}\abs{P'(\lambda)},$$
$$\abs{\p_x^{i_1}\p_y^{i_2}S(x,y)}\leq{C_S}\abs{P^{(1+i_1+i_2)}(\lambda)}$$
and
$$\abs{\sLap{S}(x,y)}\leq{C_S}(1+\lambda)^{-2s}.$$
In particular, with a possibly larger constant $C_S$, we have
$$\abs{\sLap{S}(x,y)}\leq{C_S}\abs{P'(\lambda)}.$$
\end{lem}

\begin{proof}
Again, let $(x,y)\in\overline{\Omega_j}$. Since $\abs{\bf{a}_j}^2=m_*^2$, we write
\begin{equation*}\begin{split}
S
&=\dfrac{m_*^2-\abs{\nabla\varphi}^2}{\sqrt{1+\abs{\nabla\varphi}^2}\left(c+k\sqrt{1+\abs{\nabla\varphi}^2}\right)}\\
&=\dfrac{-2{\bf{a}}_j\cdot\nabla\tilde\varphi_j-\abs{\nabla\tilde\varphi_j}^2}{\sqrt{1+\abs{\nabla\varphi}^2}\left(c+k\sqrt{1+\abs{\nabla\varphi}^2}\right)}\\
\end{split}\end{equation*}
Since the denominator is bounded between two positive numbers, it suffices to show that the dominant term of the numerator is $\abs{P'(\lambda_j)}$. We have
$$\nabla\tilde\varphi_j=({\bf{a}}_j-{\bf{a}}_{j+1})P'(\lambda_j^+)+({\bf{a}}_j-{\bf{a}}_{j-1})P'(\lambda_j^-)+\nabla\rho\ast{g_j}.$$
so
\begin{equation*}\begin{split}
&\quad\,\,-2{\bf{a}}_j\cdot\nabla\tilde\varphi_j\\
&=2(\abs{{\bf{a}}_j}^2-{\bf{a}}_j\cdot{\bf{a}}_{j+1})\abs{P'(\lambda_j^+)}+2(\abs{{\bf{a}}_j}^2-{\bf{a}}_j\cdot{\bf{a}}_{j-1})\abs{P'(\lambda_j^-)}-2{\bf{a}}_j\cdot\nabla\rho\ast{g_j}
\end{split}\end{equation*}
By Lemma \ref{varphi2}, $\abs{\nabla\rho\ast{g_j}}\leq{C_g}\abs{P'(\gamma\lambda_j)}$ and so
$$\abs{-2{\bf{a}}_j\cdot\nabla\rho\ast{g_j}}\leq2m_*{C_g}\abs{P'(\gamma\lambda_j)}.$$

Using the inequality $\abs{{\bf{a}}+{\bf{b}}+{\bf{c}}}^2\leq3(\abs{\bf{a}}^2+\abs{\bf{b}}^2+\abs{\bf{c}}^2)$, we have
\begin{equation*}\begin{split}
\abs{\nabla\tilde\varphi_j}^2
&\leq3\left(\abs{{\bf{a}}_j-{\bf{a}}_{j+1}}^2\abs{P'(\lambda_j^+)}^2+\abs{{\bf{a}}_j-{\bf{a}}_{j-1}}^2\abs{P'(\lambda_j^-)}+\abs{\nabla\rho\ast{g_j}}^2\right)\\
&\leq3\left(4\abs{P'(\lambda_j)}^2+C_g^2\abs{P'(\gamma\lambda_j)}^2\right)\\
&\leq3(4+C_g^2)\abs{P'(\lambda_j)}.
\end{split}\end{equation*}

Let $$m_0=\displaystyle\min_{1\leq{j,k}\leq{N},\,j\neq{k}}\left(\abs{{\bf{a}}_j}^2-{\bf{a}}_j\cdot{\bf{a}}_k\right).$$
We can find an $r_3>r_0$ such that for all $\lambda_j\geq{r_3}$,
$$3(4+C_g^2)\abs{P'(\lambda_j)}\leq\dfrac{m_0}{2}\quad\text{and}\quad2C_g\abs{P'(\gamma\lambda_j)}\leq\dfrac{m_0}{2}\abs{P'(\lambda_j)}.$$
This implies for all $\lambda_j\geq{r_3}$,
$$m_0\abs{P'(\lambda_j)}\leq\abs{-2{\bf{a}}_j\cdot\nabla\tilde\varphi_j-\abs{\nabla\tilde\varphi_j}^2}\leq(4m_*^2+m_0)\abs{P'(\lambda_j)}.$$
Thus, the lower and upper bounds for $S$ is established.

Next, we compute the derivatives of $S$ to yield
\begin{equation*}\begin{split}
S_x
&=-\dfrac{c\varphi_x\varphi_{xx}}{(1+\abs{\nabla\varphi})^{\frac32}}\\
S_{xx}
&=\dfrac{c}{(1+\abs{\nabla\varphi}^2)^{\frac32}}\left(\varphi_x\varphi_{xxx}+\varphi_{xx}^2-\dfrac{3\varphi_x^2\varphi_{xx}^2}{1+\abs{\nabla\varphi}^2}\right)\\
S_{xy}
&=\dfrac{3c\varphi_x\varphi_y\varphi_{xx}\varphi_{yy}}{(1+\abs{\nabla\varphi})^{\frac52}}.
\end{split}\end{equation*}
Using the Lemmas \ref{varphi1}, \ref{P} and \ref{varphi3}, it is clear that
\begin{equation*}\begin{split}
\abs{S_x}
&\leq{c}{m_*}{C_\varphi}P''(\lambda)\\
\abs{S_{xx}}
&\leq{c}\left({m_*}{C_\varphi}\abs{P^{(3)}(\lambda)}+(1+3m_*^2)C_\varphi^2P''(\lambda)^2\right)\\
&\leq\tilde{C}_2(c,m_*,\tilde\rho,\gamma)\abs{P^{(3)}(\lambda)}\\
\abs{S_{xy}}
&\leq3c{m_*^2}C_\varphi^2P''(\lambda)^2\\
&\leq\tilde{C}_2(c,m_*,\tilde\rho,\gamma)\abs{P^{(3)}(\lambda)}
\end{split}\end{equation*}
Hence the $C^2$ norm of $S$ is controlled.

To estimate the $s$-Laplacian, we use Corollary \ref{decay2} to get
\begin{equation*}\begin{split}
\abs{\sLap{S}(x,y)}
&\leq{C_\Delta}\left(\norm[\dot{C}^2(B_1(x,y))]{S}+\norm[L^\infty(\R^2)]{S}\right)\\
&\leq{C_\Delta}\tilde{C}_3\left(\abs{P^{(3)}(\lambda-1)}+c-k\right)\\
&\leq\tilde{C}_4
\end{split}\end{equation*}
for $\lambda\leq2r_0$, and
\begin{equation*}\begin{split}
&\quad\,\,\abs{\sLap{S}(x,y)}\\
&\leq{C_\Delta}\left(\norm[\dot{C}^2(B_{\frac{\abs{(x,y)}}{2}}(x,y))]{S}+\norm[\dot{C}^1(B_{\frac{\abs{(x,y)}}{2}}(x,y))]{S}+\norm[L^\infty(\R^2)]{S}\abs{(x,y)}^{-2s}\right)\\
&\leq{C_\Delta}\tilde{C}_3\left(4\abs{P^{(3)}\left(\dfrac{\lambda}{2}\right)}+2\abs{P''\left(\dfrac{\lambda}{2}\right)}+(c-k)\lambda^{-2s}\right)\\
&\leq\tilde{C}_5\lambda^{-2s}
\end{split}\end{equation*}
for $\lambda\geq2r_0$, by Lemma \ref{P}. Thus,
$$\abs{\sLap{S}(x,y)}\leq\tilde{C}_6(1+\lambda)^{-2s}.$$
The last assertion follows from Lemma \ref{P}.
\end{proof}

\medskip
\section{The super-solution}
\label{sect supersol}

Let $\alpha,\eps\in(0,1)$ be small parameters (which will be chosen according to \eqref{param eps} and \eqref{param alpha}) and write
\begin{equation*}
V(x,y,z)=\Phi\left(\hat\mu\right)+\eps{S}(\alpha{x},\alpha{y}),
\end{equation*}
where $S$ is defined in \eqref{def S} and
\begin{equation}\label{def mu}
\hat\mu=\dfrac{z-\alpha^{-1}\varphi(\alpha{x},\alpha{y})}{\sqrt{1+\abs{\nabla\varphi(\alpha{x},\alpha{y})}^2}}.
\end{equation}

Introducing the rescaled one-dimensional profile
\begin{equation}\label{def phi alpha}
\Phi_\alpha(\mu)=\Phi(\alpha^{-1}\mu),
\end{equation}
we can also write
$$V(x,y,z)=\Phi_\alpha\left(\bar\mu(\alpha{x},\alpha{y},\alpha{z})\right)+\eps{S}(\alpha{x},\alpha{y}),$$
where
$$\bar\mu(x,y,z)=\dfrac{z-\varphi(x,y)}{\sqrt{1+\abs{\nabla\varphi(x,y)}^2}}.$$

We will prove that $V$ is a super-solution when $\eps$ and $\alpha$ are sufficiently small.

Observing that $\bar\mu(x,y,z)$ is a linear in $z$ and ``almost linear'' in $x$ and $y$ as $\lambda(x,y)\to\infty$, it is tempting to compare $\sLap(\Phi_\alpha(\bar\mu))$ with $\sLapx{\Phi_\alpha}(\bar\mu)$, according to the chain rule, Lemma \ref{chain}. It turns out that their difference
$$R_1(x,y,z;\alpha)=\sLap(\Phi_\alpha(\bar\mu(x,y,z)))-\sLapx\Phi_\alpha(\bar\mu(x,y,z))$$
decays at the right order as $\lambda\to\infty$. This crucial estimate is stated as follows.

\begin{prop}[Main estimate]\label{R}
There exists a constant $C_{R_1}=C_{R_1}(s,\tilde\rho,\gamma)$ such that for any $(x,y,z)\in\R^3$, we have
$$\abs{R_1(x,y,z;\alpha)}\leq{C_{R_1}}\alpha^{-s}\left(1+\lambda(x,y)\right)^{-2s},$$
uniformly in $z$. We can also write
$$\abs{R_1(x,y,z;\alpha)}\leq{C_{R_1}}\alpha^{-s}\abs{P'(\lambda(x,y))}.$$
\end{prop}

\begin{remark}
In fact, denoting $\varphi^\alpha(x,y)=\alpha^{-1}\varphi(\alpha{x},\alpha{y})$, the first term of $V$ can be expressed as
$$\Phi\left(\dfrac{z-\varphi^\alpha(x,y)}{\sqrt{1+\abs{\nabla\varphi^\alpha(x,y)}^2}}\right).$$
In this way, the estimate of $R_1$ gives the first term of the fractional Laplacian in the Fermi coordinates, roughly as $\sLap=(-\p_\mu^2)^s+\cdots$ if $\mu$ denotes the signed distance to the surface $z=\varphi^\alpha(x,y)$, as $\alpha\to0$.
\end{remark}

It is helpful to keep in mind that the decay in $x$ and $y$ comes from the second and third derivatives of $\varphi$. The uniformity in $z$ will follow from the facts that $z$ can be written in terms of $\bar\mu(x,y,z)$ and quantities like $\bar\mu^i\Phi_\alpha^{(j)}(\bar\mu)$ for $0\leq{i}\leq{j}$, $1\leq{j}\leq2$ are uniformly bounded.

\begin{proof}[Proof of Proposition \ref{R}]
We will prove the estimate pointwisely in $(x,y)\in\R^2$. To simplify the notations, $\bar\mu$ and $\varphi$ and their derivatives will be evaluated at $(x,y,z)$ or $(x,y)$ unless otherwise specified.

First we write down the difference. By definition, we have
\begin{equation*}\begin{split}
&\quad\;\sLap(\Phi_\alpha(\bar\mu))\\
&=C_{3,s}\PV\int_{\R^3}\dfrac{\Phi_\alpha(\bar\mu)-\Phi_\alpha(\bar\mu(x+\xi,y+\eta,z+\zeta))}{(\xi^2+\eta^2+\zeta^2)^{\frac32+s}}\,d{\xi}d{\eta}d{\zeta}\\
&=C_{3,s}\PV\int_{\R^3}\!\left[\Phi_\alpha(\bar\mu)-\Phi_\alpha\left(\dfrac{z+\zeta-\varphi(x+\xi,y+\eta)}{\sqrt{1+\abs{\nabla\varphi(x+\xi,y+\eta)}^2}}\right)\right]\dfrac{d{\xi}d{\eta}d{\zeta}}{(\xi^2+\eta^2+\zeta^2)^{\frac32+s}}
\end{split}\end{equation*}

On the other hand, by Corollary \ref{1var}, we have
\begin{equation*}\begin{split}
\sLapx\Phi_\alpha(\bar\mu)
&=C_{1,s}\PV\int_{\R}\!\dfrac{\Phi_\alpha(\bar\mu)-\Phi_\alpha(\bar\mu+\mu)}{\abs{\mu}^{1+2s}}\,d\mu\\
&=C_{3,s}\PV\int_{\R^3}\!\dfrac{\Phi_\alpha(\bar\mu)-\Phi_\alpha(\bar\mu+\mu)}{(\mu^2+\nu^2+\theta^2)^{\frac32+s}}\,d{\mu}d{\nu}d{\theta}
\end{split}\end{equation*}
In order to rotate the axes to the tangent and normal directions of the graph $z=\varphi(x,y)$, we wish to find an orthogonal linear transformation
$\begin{pmatrix}\mu\\\nu\\\theta\end{pmatrix}=T\begin{pmatrix}\zeta\\\xi\\\eta\end{pmatrix}$
that sends $\mu$ to
$\dfrac{\zeta-\varphi_{x}\xi-\varphi_{y}\eta}{\sqrt{1+\varphi_x^2+\varphi_y^2}}$.
To this end, we apply Gram-Schmidt process to the basis
$$\dfrac{(1,-\varphi_x,-\varphi_y)}{\sqrt{1+\varphi_x^2+\varphi_y^2}},\quad\dfrac{(\varphi_x,1,0)}{\sqrt{1+\varphi_x^2}},\quad\dfrac{(\varphi_y,0,1)}{\sqrt{1+\varphi_y^2}}$$
of $\R^3$ to obtain the orthonormal vectors
$$\dfrac{(1,-\varphi_x,-\varphi_y)}{\sqrt{1+\varphi_x^2+\varphi_y^2}},\quad\dfrac{(\varphi_x,1,0)}{\sqrt{1+\varphi_x^2}},\quad\dfrac{(\varphi_y,-\varphi_x\varphi_y,1+\varphi_x^2)}{\sqrt{1+\varphi_x^2}\sqrt{1+\varphi_x^2+\varphi_y^2}}.$$
It is easy to check that
\begin{equation*}
T=\begin{pmatrix}
\dfrac{1}{\sqrt{1+\varphi_x^2+\varphi_y^2}}&\dfrac{-\varphi_x}{\sqrt{1+\varphi_x^2+\varphi_y^2}}&\dfrac{-\varphi_y}{\sqrt{1+\varphi_x^2+\varphi_y^2}}\\
\dfrac{\varphi_x}{\sqrt{1+\varphi_x^2}}&\dfrac{1}{\sqrt{1+\varphi_x^2}}&0\\
\dfrac{\varphi_y}{\sqrt{1+\varphi_x^2}\sqrt{1+\varphi_x^2+\varphi_y^2}}&\dfrac{-\varphi_x\varphi_y}{\sqrt{1+\varphi_x^2}\sqrt{1+\varphi_x^2+\varphi_y^2}}&\sqrt{\dfrac{1+\varphi_x^2}{1+\varphi_x^2+\varphi_y^2}}
\end{pmatrix}
\end{equation*}
indeed satisfies the required condition.

Under this transformation, we have
\begin{equation*}\begin{split}
&\quad\;\sLapx\Phi_\alpha(\bar\mu)\\
&=C_{3,s}\PV\int_{\R^3}\left[\Phi_\alpha(\bar\mu)-\Phi_\alpha\left(\bar\mu+\dfrac{\zeta-\varphi_{x}\xi-\varphi_{y}\eta}{\sqrt{1+\abs{\nabla\varphi}^2}}\right)\right]\dfrac{d{\zeta}d{\xi}d{\eta}}{(\zeta^2+\xi^2+\eta^2)^{\frac32+s}}\\
&=C_{3,s}\PV\int_{\R^3}\left[\Phi_\alpha(\bar\mu)-\Phi_\alpha\left(\dfrac{z+\zeta-\varphi-\varphi_{x}\xi-\varphi_{y}\eta}{\sqrt{1+\abs{\nabla\varphi}^2}}\right)\right]\dfrac{d{\zeta}d{\xi}d{\eta}}{(\zeta^2+\xi^2+\eta^2)^{\frac32+s}}
\end{split}\end{equation*}
The difference is therefore
\begin{equation*}\begin{split}
&\quad\,\,\sLap(\Phi_\alpha(\bar\mu))-\sLapx\Phi_\alpha(\bar\mu)\\
&=C_{3,s}\PV\int_{\R^3}\!\Bigg[-\Phi_\alpha\left(\dfrac{z+\zeta-\varphi(x+\xi,y+\eta)}{\sqrt{1+\abs{\nabla\varphi(x+\xi,y+\eta)}^2}}\right)\\
&\qquad\qquad\qquad\qquad+\Phi_\alpha\left(\dfrac{z+\zeta-\varphi-\varphi_{x}\xi-\varphi_{y}\eta}{\sqrt{1+\abs{\nabla\varphi}^2}}\right)\Bigg]\,\dfrac{d{\xi}d{\eta}d{\zeta}}{(\xi^2+\eta^2+\zeta^2)^{\frac32+s}}\\
\end{split}\end{equation*}

We will estimate this integral in
$$D_\alpha=\set{(\xi,\eta,\zeta)\in\R^3\,\Bigg|\,\sqrt{\xi^2+\eta^2+\zeta^2}<\dfrac{\sqrt{\alpha}}{2}\left(C_P\abs{P'(\lambda)}\right)^{-\frac{1}{{2s}}}}$$
and its complement $D_\alpha^c$ in $\R^3$
and write
\begin{equation*}\begin{split}
\sLap(\Phi_\alpha(\bar\mu))-\sLapx\Phi_\alpha(\bar\mu)=C_{3,s}(J_1+J_2),
\end{split}\end{equation*}
where
\begin{equation}\label{J1}\begin{split}
J_1
&=\int_{D_\alpha}\!\Bigg[-\Phi_\alpha\left(\dfrac{z+\zeta-\varphi(x+\xi,y+\eta)}{\sqrt{1+\abs{\nabla\varphi(x+\xi,y+\eta)}^2}}\right)\\
&\qquad\qquad\quad+\Phi_\alpha\left(\dfrac{z+\zeta-\varphi-\varphi_{x}\xi-\varphi_{y}\eta}{\sqrt{1+\abs{\nabla\varphi}^2}}\right)\\
&\qquad\qquad\quad-\bar\mu\Phi_\alpha'(\bar\mu)\dfrac{\varphi_{x}\varphi_{xx}\xi+\varphi_{y}\varphi_{yy}\eta}{1+\abs{\nabla\varphi}^2}\Bigg]\,\dfrac{d{\xi}d{\eta}d{\zeta}}{(\xi^2+\eta^2+\zeta^2)^{\frac32+s}}\\
\end{split}\end{equation}
and
\begin{equation*}\begin{split}
J_2
&=\int_{D_\alpha^c}\!\Bigg[-\Phi_\alpha\left(\dfrac{z+\zeta-\varphi(x+\xi,y+\eta)}{\sqrt{1+\varphi'(x+\xi,y+\eta)^2}}\right)\\
&\qquad\qquad+\Phi_\alpha\left(\dfrac{z+\zeta-\varphi-\varphi_{x}\xi-\varphi_{y}\eta}{\sqrt{1+\abs{\nabla\varphi}^2}}\right)\Bigg]\,\dfrac{d{\xi}d{\eta}d{\zeta}}{(\xi^2+\eta^2+\zeta^2)^{\frac32+s}}\\
\end{split}\end{equation*}
Note that the extra term in $J_1$, being odd in $\xi$ and $\eta$, is inserted to get rid of the principal value.

Since $\abs{\Phi_\alpha}\leq1$, it is easy to see that
\begin{equation*}\begin{split}
\abs{J_2}
&\leq2\int_{D_\alpha^c}\!\,\dfrac{d{\xi}d{\eta}d{\zeta}}{(\xi^2+\eta^2+\zeta^2)^{\frac32+s}}\\
&=8\pi\int_{\frac{\sqrt{\alpha}}{2}\left(C_P\abs{P'(\lambda)}\right)^{-\frac{1}{{2s}}}}^{\infty}\dfrac{r^2\,dr}{r^{3+2s}}\\
&=\dfrac{4\pi}{s}\left(\dfrac{\sqrt{\alpha}}{2}\left(C_P\abs{P'(\lambda)}\right)^{-\frac{1}{{2s}}}\right)^{-2s}\\
&=\dfrac{2^{2+2s}\pi}{s}C_P\alpha^{-s}\abs{P'(\lambda)}\\
&\leq\dfrac{2^{2+2s}\pi}{s}C_P^2\alpha^{-s}(1+\lambda)^{-2s}
\end{split}\end{equation*}
by using Lemma \ref{P}.

To estimate $J_1$, we let $F(t)=F(t,\xi,\eta,\zeta;x,y,z,\alpha)=-\Phi_\alpha(\mu^*(t))$, where
\begin{equation*}\begin{split}
\mu^*(t)
&=\mu^*(t,\xi,\eta,\zeta;x,y,z)\\
&=\dfrac{z+\zeta-(1-t)(\varphi+\varphi_{x}\xi+\varphi_{y}\eta)-t\varphi(x+\xi,y+\eta)}{\sqrt{1+\abs{\nabla\varphi(x+t\xi,y+t\eta)}^2}}.
\end{split}\end{equation*}
Let us introduce the shorthand notation
$$\varphi^t=\varphi(x+t\xi,y+t\eta),\quad(\nabla\varphi)^t=\nabla\varphi(x+t\xi,y+t\eta)$$
and similarly for other derivatives of $\varphi$ for $t\in[0,1]$. So we can rewrite
$$\mu^*(t)=\dfrac{z+\zeta-(1-t)(\varphi+\varphi_{x}\xi+\varphi_{y}\eta)-t\varphi^1}{\sqrt{1+\abs{(\nabla\varphi)^t}^2}}.$$
We will use the second order Taylor expansion to write the first two terms in the integrand (considering the denominator as a weight) as
$$F(1)-F(0)=F'(0)+\int_0^1(1-t)F''(t)\,dt.$$
The derivative of $\mu^*(t)$ can be written as
$$\mu_t^*(t)=-A(t)-B(t)\mu^*(t),$$
where
\begin{equation*}\begin{split}
A(t)=A(t,\xi,\eta;x,y)
&=\dfrac{\varphi^1-\varphi-\varphi_{x}\xi-\varphi_{y}\eta}{\sqrt{1+\abs{(\nabla\varphi)^t}^2}}\\
&=\dfrac{\displaystyle\int_0^1(1-t')\left(\varphi_{xx}^{t'}\xi^2+2\varphi_{xy}^{t'}\xi\eta+\varphi_{yy}^{t'}\eta^2\right)\,dt'}{\sqrt{1+\abs{(\nabla\varphi)^t}^2}}\\
B(t)=B(t,\xi,\eta;x,y)
&=\dfrac{(\varphi_x^t\varphi_{xx}^t+\varphi_y^t\varphi_{xy}^t)\xi+(\varphi_x^t\varphi_{xy}^t+\varphi_y^t\varphi_{yy}^t)\eta}{1+\abs{(\nabla\varphi)^t}^2}\\
&=\dfrac{\left(\abs{(\nabla\varphi)^t}^2\right)_x\xi+\left(\abs{(\nabla\varphi)^t}^2\right)_y\eta}{2\left(1+\abs{(\nabla\varphi)^t}^2\right)}.
\end{split}\end{equation*}
Note that $A_t(t)=-A(t)B(t)$ and
\begin{equation*}\begin{split}
B_t(t)
&=\left(\dfrac{\left(\abs{(\nabla\varphi)^t}^2\right)_{xx}}{2\left(1+\abs{(\nabla\varphi)^t}^2\right)}-\left(\dfrac{\left(\abs{(\nabla\varphi)^t}^2\right)_x}{1+\abs{(\nabla\varphi)^t}^2}\right)^2\right)\xi^2\\
&\qquad+\left(\dfrac{\left(\abs{(\nabla\varphi)^t}^2\right)_{xy}}{1+\abs{(\nabla\varphi)^t}^2}-\dfrac{2\left(\abs{(\nabla\varphi)^t}^2\right)_x\left(\abs{(\nabla\varphi)^t}^2\right)_y}{\left(1+\abs{(\nabla\varphi)^t}^2\right)^2}\right)\xi\eta\\
&\qquad+\left(\dfrac{\left(\abs{(\nabla\varphi)^t}^2\right)_{yy}}{2\left(1+\abs{(\nabla\varphi)^t}^2\right)}-\left(\dfrac{\left(\abs{(\nabla\varphi)^t}^2\right)_y}{1+\abs{(\nabla\varphi)^t}^2}\right)^2\right)\eta^2,
\end{split}\end{equation*}
each of whose term containing at least a third derivative or a product of two second derivatives of $\varphi$, evaluated at $(x+t\xi,y+t\eta)$.

We compute
\begin{equation*}\begin{split}
F'
&=-\Phi_\alpha'(\mu^*)\mu_t^*,\\
F''
&=-\Phi_\alpha''(\mu^*)(\mu_t^*)^2-\Phi_\alpha'(\mu^*)\mu_{tt}^*,\\
(\mu_t^*)^2
&=(A+B\mu^*)^2\\
&=A^2+2AB\mu^*+B^2(\mu^*)^2,\\
\mu_{tt}^*
&=-A_t-B_t\mu^*-B\mu_t^*\\
&=AB-B_t\mu^*+B(A+B\mu^*)\\
&=2AB+(B^2-B_t)\mu^*,\\
F'(0)
&=A(0)\Phi_\alpha'(\mu^*(0))+B(0)\mu^*(0)\Phi_\alpha'(\mu^*(0)).
\end{split}\end{equation*}
Observe that the extra term in the integrand is chosen in such a way
$$-\bar\mu\Phi_\alpha'(\bar\mu)\dfrac{(\varphi_{x}\varphi_{xx}+\varphi_{y}\varphi_{xy})\xi+(\varphi_{x}\varphi_{xy}+\varphi_{y}\varphi_{yy})\eta}{1+\abs{\nabla\varphi}^2}=-B(0)\bar\mu\Phi_\alpha'(\bar\mu),$$
that is similar to the second term in $F'(0)$. In fact, if we let $$\tilde\mu(t)=\tilde\mu(t,\xi,\eta,\zeta;x,y,z)=\bar\mu+t(\mu^*(0)-\bar\mu)$$
and
\begin{equation*}\begin{split}
G(t)
&=G(t,\xi,\eta,\zeta;x,y,z,\alpha)\\
&=\tilde\mu(t)\Phi_\alpha'(\tilde\mu(t))\\
&=(\bar\mu+t(\mu^*(0)-\bar\mu))\Phi_\alpha'(\bar\mu+t(\mu^*(0)-\bar\mu)),\\
\end{split}\end{equation*}
then
\begin{equation*}\begin{split}
&\quad\,\,\mu^*(0)\Phi_\alpha'(\mu^*(0))-\bar\mu\Phi_\alpha'(\bar\mu)\\
&=G(1)-G(0)\\
&=\int_0^1G'(t)\,dt\\
&=(\mu^*(0)-\bar\mu)\int_0^1\left(\Phi_\alpha'(\tilde\mu(t))+\tilde\mu(t)\Phi_\alpha''(\tilde\mu(t))\right)\,dt\\
&=\dfrac{\zeta-\varphi_{x}\xi-\varphi_{y}\eta}{\sqrt{1+\abs{\nabla\varphi}^2}}\int_0^1\left(\Phi_\alpha'(\tilde\mu(t))+\tilde\mu(t)\Phi_\alpha''(\tilde\mu(t))\right)\,dt
\end{split}\end{equation*}
and hence
\begin{equation*}\begin{split}
&\quad\,\,F'(0)-B(0)\bar\mu\Phi_\alpha'(\bar\mu)\\
&=A(0)\Phi_\alpha'(\mu^*(0))\\
&\qquad+B(0)\dfrac{\zeta-\varphi_{x}\xi-\varphi_{y}\eta}{\sqrt{1+\abs{\nabla\varphi}^2}}\int_0^1\!\left(\Phi_\alpha'(\tilde\mu(t))+\tilde\mu(t)\Phi_\alpha''(\tilde\mu(t))\right)\,dt.
\end{split}\end{equation*}
The integrand in \eqref{J1} now has the expression
\begin{equation*}\begin{split}
&\quad\,\,-\Phi_\alpha\left(\dfrac{z+\zeta-\varphi^1}{\sqrt{1+\abs{(\nabla\varphi)^1}^2}}\right)+\Phi_\alpha\left(\dfrac{z+\zeta-\varphi-\varphi_{x}\xi-\varphi_{y}\eta}{\sqrt{1+\abs{\nabla\varphi}^2}}\right)\\
&\qquad\,\,-\bar\mu\Phi_\alpha'(\bar\mu)\dfrac{\varphi_{x}\varphi_{xx}\xi+\varphi_{y}\varphi_{yy}\eta}{1+\abs{\nabla\varphi}^2}\\
&=F(1)-F(0)-B(0)\bar\mu\Phi_\alpha'(\bar\mu)\\
&=F'(0)-B(0)\bar\mu\Phi_\alpha'(\bar\mu)+\int_0^1\!(1-t)F''(t)\,dt\\
&=A(0)\Phi_\alpha'(\mu^*(0))+B(0)\dfrac{\zeta-\varphi_{x}\xi-\varphi_{y}\eta}{\sqrt{1+\abs{\nabla\varphi}^2}}\int_0^1\!(\Phi_\alpha'(\tilde\mu)+\tilde\mu\Phi_\alpha''(\tilde\mu))\,dt\\
&\quad\,\,+\int_0^1\!(1-t)\Big[-\Phi_\alpha''(\mu^*)\left(A^2+2AB\mu^*+B^2(\mu^*)^2\right)\\
&\qquad\qquad\qquad\qquad-\Phi_\alpha'(\mu^*)\left(2AB+\left(B^2-B_t\right)\mu^*\right)\Big]\,dt
\end{split}\end{equation*}

In $D_\alpha\subset{D_1}$, we have $$\sqrt{\xi^2+\eta^2+\zeta^2}\leq\dfrac{1}{2}\left(C_P\abs{P'(\lambda)}\right)^{-\frac{1}{{2s}}}\leq\dfrac{1+\lambda}{2}$$
and hence
$$1+\lambda(x+t\xi,y+t\eta)\geq1+\lambda(x,y)-t\sqrt{\xi^2+\eta^2}\geq\dfrac{1+\lambda(x,y)}{2}.$$
We use this to bound, using Lemma \ref{varphi3} and Lemma \ref{P},
\begin{equation*}\begin{split}
\abs{\p_x^{i_1}\p_y^{i_2}\varphi(x+t\xi,y+t\eta)}
&\leq{C_\varphi}\abs{P^{(i_1+i_2)}(\lambda(x+t\xi,y+t\eta))}\\
&\leq{C_\varphi}{C_P}(1+\lambda(x+t\xi,y+t\eta))^{-{2s}+1-i_1-i_2}\\
&\leq{C_\varphi}{C_P}2^{{2s}-1+i_1+i_2}(1+\lambda(x,y))^{-{2s}+1-i_1-i_2}\\
&\leq16{C_\varphi}{C_P}(1+\lambda)^{-{2s}+1-i_1-i_2}.
\end{split}\end{equation*}
This gives an estimate for $A$, $B$ and $B_t$
\begin{equation*}\begin{split}
\abs{A(t)}
&\leq8{C_\varphi}{C_P}\left(\abs{\xi}^2+2\abs{\xi}\abs{\eta}+\abs{\eta}^2\right)(1+\lambda)^{-{2s}-1},\\
\abs{B(t)}
&\leq32{m_*}{C_\varphi}{C_P}(\abs{\xi}+\abs{\eta})(1+\lambda)^{-{2s}-1},\\
\abs{B_t(t)}
&\leq\tilde{C}_7(s,m_*,\tilde\rho,\gamma)\left(\abs{\xi}^2+\abs{\xi}\abs{\eta}+\abs{\eta}^2\right)(1+\lambda)^{-{2s}-2},
\end{split}\end{equation*}
which are all uniform in $t\in[0,1]$. Note that for the estimate of $B_t$ we have used \eqref{r0}. Denoting
$$r=\sqrt{\xi^2+\eta^2+\zeta^2}<1\qquad\text{and}\qquad\norm[\infty]{A}=\displaystyle\sup_{t\in[0,1]}\abs{A(t)},$$
similarly for $B$ and $B_t$, we can find a constant $C_{AB}=C_{AB}(s,m_*,\tilde\rho,\gamma)$ such that
\begin{equation*}\begin{split}
\norm[\infty]{A}
&\leq{C}_{AB}r^2(1+\lambda)^{-{2s}-1},\\
2\norm[\infty]{B}
&\leq{C}_{AB}r(1+\lambda)^{-{2s}-1},\\
\norm[\infty]{A}^2
&\leq{C}_{AB}r^4(1+\lambda)^{-{4s}-2},\\
4\norm[\infty]{A}\norm[\infty]{B}
&\leq{C}_{AB}r^2(1+\lambda)^{-{4s}-1},\\
2\norm[\infty]{B}^2
&\leq{C}_{AB}r^2(1+\lambda)^{-{4s}-2},\\
\norm[\infty]{B_t}
&\leq{C}_{AB}r^2(1+\lambda)^{-{2s}-2}.
\end{split}\end{equation*}
Note also that by Cauchy-Schwarz inequality,
$$\dfrac{\zeta-\varphi_{x}\xi-\varphi_{y}\eta}{\sqrt{1+\abs{\nabla\varphi}^2}}\leq\sqrt{\xi^2+\eta^2+\zeta^2}=r.$$

Denoting $\norm[\infty]{\mu^i\Phi_\alpha^{(j)}}=\displaystyle\sup_{\mu\in\R}\,\abs{\mu^i\Phi_\alpha^{(j)}(\mu)}$ for $0\leq{i}\leq{j}\leq2$ and using Corollary \ref{Phi''}, we can estimate $J_1$ as
\begin{equation*}\begin{split}
\abs{J_1}
&\leq\int_{D_\alpha}\!\Big[\abs{A(0)}\norm[\infty]{\Phi_\alpha'}+\abs{B(0)}r\left(\norm[\infty]{\Phi_\alpha'}+\norm[\infty]{\mu\Phi_\alpha''}\right)\\
&\qquad\qquad+\norm[\infty]{\Phi_\alpha''}\norm[\infty]{A}^2+2\norm[\infty]{A}\norm[\infty]{B}(\norm[\infty]{\Phi_\alpha'}+\norm[\infty]{\mu\Phi_\alpha''})\\
&\qquad\qquad+\norm[\infty]{B}^2(\norm[\infty]{\mu\Phi_\alpha'}+\norm[\infty]{\mu^2\Phi_\alpha''})\\
&\qquad\qquad+\norm[\infty]{B_t}\norm[\infty]{\mu\Phi_\alpha'}\Big]\,\dfrac{d{\xi}d{\eta}d{\zeta}}{(\xi^2+\eta^2+\zeta^2)^{\frac32+s}}\\
&\leq4\pi{C_{AB}}{C_\Phi}(1+\lambda)^{-{2s}-1}\int_{0}^{\frac{\sqrt{\alpha}}{2}\left(C_P\abs{P'(\lambda)}\right)^{-\frac{1}{{2s}}}}\!\Bigg[\alpha^{-1}(r^2+r^2)\\
&\qquad\qquad+\alpha^{-2}r^4(1+\lambda)^{-{2s}-1}+\alpha^{-1}r^2(1+\lambda)^{-{2s}}+r^2(1+\lambda)^{-{2s}-1}\\
&\qquad\qquad+r^2(1+\lambda)^{-1}\Bigg]\dfrac{r^2\,dr}{r^{3+2s}}\\
&\leq4\pi{C_{AB}}{C_\Phi}(1+\lambda)^{-{2s}-1}\int_{0}^{\frac{\sqrt{\alpha}}{2}\left(C_P\abs{P'(\lambda)}\right)^{-\frac{1}{{2s}}}}\!\Big[\alpha^{-2}(1+\lambda)^{-{2s}-1}r^{3-2s}\\
&\hspace{7.5cm}+5\alpha^{-1}r^{1-2s}\Big]\,dr\\
&=\dfrac{4\pi{C_{AB}}{C_\Phi}}{4-2s}\alpha^{-2}(1+\lambda)^{-{4s}-2}\left(\dfrac{\sqrt\alpha}{2}\right)^{4-2s}(C_P\abs{P'(\lambda)})^{-\frac{4-2s}{{2s}}}\\
&\qquad+\dfrac{20\pi{C_{AB}}{C_\Phi}}{2-2s}\alpha^{-1}(1+\lambda)^{-{2s}-1}\left(\dfrac{\sqrt\alpha}{2}\right)^{2-2s}(C_P\abs{P'(\lambda)})^{-\frac{2-2s}{{2s}}}\\
&=\dfrac{\pi{C_{AB}}{C_\Phi}}{2^{3-2s}(2-s)}\alpha^{-s}(1+\lambda)^{-{6s}+2}
+\dfrac{5\pi{C_{AB}}{C_\Phi}}{2^{1-2s}(1-s)}\alpha^{-s}(1+\lambda)^{-{4s}+1}\\
&=\dfrac{5\pi{C_{AB}}{C_\Phi}}{2^{1-2s}(1-s)}\alpha^{-s}(1+\lambda)^{-{4s}+1}\left(1+\dfrac{1-s}{20(2-s)}(1+\lambda)^{-{2s}+1}\right)\\
&\leq\dfrac{2^{2s}5\pi{C_{AB}}{C_\Phi}}{1-s}\alpha^{-s}(1+\lambda)^{-{4s}+1}
\end{split}\end{equation*}
The proof is completed by taking $$C_{R_1}=\dfrac{2^{2+2s}\pi}{s}C_P^2+\dfrac{2^{2s}5\pi}{1-s}{C_{AB}}{C_\Phi}.$$
\end{proof}


We can now prove

\begin{prop}\label{supersol}
There exist small parameters $\eps$ and $\alpha$ such that $V(x,y,z)$ is a super-solution of (\ref{3D}). In fact, $V\in{C}_s^{2}(\R^3)$ and
$${\mathcal L}[V]>0 \quad\text{and}\quad v<V \quad\text{in}~\R^3.$$
\end{prop}

\begin{proof}
We recall the calculations in Section \ref{sect motiv} gives
$$\mathcal{L}[V]=S(\alpha{x},\alpha{y})\left(-\Phi'(\hat\mu)-\eps\int_0^1\!f'(\Phi(\hat\mu)+t\eps{S})\,dt+\alpha^{2s}\dfrac{\bar{R}(\alpha{x},\alpha{y})}{S(\alpha{x},\alpha{y})}\right),$$
where $S$ is defined by \eqref{def S} and $\bar{R}=R_1+R_2$ with
\begin{equation*}\begin{split}
R_1
&=\sLap(\Phi_\alpha(\bar\mu(\alpha{x},\alpha{y},\alpha{z})))-(\sLapx\Phi_\alpha)(\bar\mu(\alpha{x},\alpha{y},\alpha{z})),\\
R_2
&=\eps(\sLap{S})(\alpha{x},\alpha{y}).
\end{split}\end{equation*}
By Proposition \ref{R} and Lemma \ref{S}, there is a constant $C_R=C_R(s,c,m_*,\tilde\rho,\gamma)$ such that
$${\mathcal{L}}[V]\geq{S}(\alpha{x},\alpha{y})\left(-\Phi'(\hat\mu)-\eps\int_0^1\!f'(\Phi(\hat\mu)+t\eps{S})\,dt-C_R\alpha^{s}\right).$$
The rest of the proof is similar to \cite{taniguchi1}. Here we still include it for the sake of completeness.

Since $f$ is $C^1$ and $f'(\pm1)<0$, there exists constants $\delta_*\in(0,1/4)$ and $\kappa_1>0$ such that
$$-f'(t)>\kappa_1\quad\text{for}\quad1-\abs{t}<2\delta_*.$$
Write $\kappa_2=\norm[{L^\infty(-1-\delta_*,1+\delta_*)}]{f'}$ and
$$\Phi_*=\min\set{-\Phi'(\mu)\mid-1+\delta_*\leq\Phi(\mu)\leq1-\delta_*}.$$
Also, in view of Lemma \ref{varphi3} and Lemma \ref{S},
$$\dfrac{\varphi(x,y)-h(x,y)}{S(x,y)}\geq\dfrac{{C_\varphi^{-1}}P(\lambda)}{C_S\abs{P'(\lambda)}}\geq{C_\varphi^{-1}}{C_S^{-1}}{C_P^{-2}}(1+\lambda)^{-1}$$
for all $(x,y)\in\R^2$, so we can write
$$\omega=\min_{(x,y)\in\R^2}\dfrac{\varphi(x,y)-h(x,y)}{S(x,y)}\geq{C_\varphi^{-1}}{C_S^{-1}}{C_P^{-2}}>0.$$
By the decay of $\Phi'(\mu)$, we can find constants $\tilde{C}_\Phi$ and $r_4$ such that
$$\abs{\Phi'(\mu)}\leq\dfrac{\tilde{C}_\Phi}{\abs{\mu}^{1+2s}}\quad\text{for}\quad\abs{\mu}\geq{r_4}.$$
Hence
$$\abs{\mu\Phi'(\mu)}\leq\dfrac{\tilde{C}_\Phi}{\abs{\mu}^{2s}}\quad\text{for}\quad\abs{\mu}\geq{r_4}.$$

Choose
\begin{equation}\label{param eps}
\eps<\min\set{\dfrac12,\dfrac{\delta_*}{c-k},\dfrac{\Phi_*}{3\kappa_2}}
\end{equation}
and then choose
\begin{equation}\label{param alpha}
\alpha<\min\set{\dfrac12,\left(\dfrac{\eps\kappa_1}{2C_R}\right)^{\frac1s},\left(\dfrac{\Phi_*}{3C_R}\right)^{\frac1s},\dfrac{k\omega}{2r_4},\dfrac{k\omega}{2}\left(\dfrac{k\eps}{2\tilde{C}_\Phi}\right)^{\frac{1}{2s}}}.
\end{equation}
We consider two cases.

Case 1: $1-\abs{\Phi(\hat\mu)}<\delta_*$. Then for $0\leq{t}\leq1$, $\abs{t\eps{S}}\leq\eps(c-k)\leq\delta_*$ and so $1-(\Phi(\hat\mu)+t\eps{S})<2\delta_*$. Since $-\Phi'(\hat\mu)>0$, we have
$${\mathcal{L}}[V]\geq{S}(\alpha{x},\alpha{y})(\eps\kappa_1-C_R\alpha^s)\geq\dfrac{\eps\kappa_1}{2}S(\alpha{x},\alpha{y})>0.$$

Case 2: $-1+\delta_*\leq\Phi(\hat\mu)\leq1-\delta_*$. Then
$${\mathcal{L}}[V]\geq{S}(\alpha{x},\alpha{y})\left(\Phi_*-\eps\kappa_2-C_R\alpha^s\right)\geq\dfrac{\Phi_*}{3}S(\alpha{x},\alpha{y})>0.$$

Therefore, ${\mathcal{L}}[V]>0$ on $\R^3$, that is, $V$ is a super-solution. Next we prove that $v<V$ on $\R^3$, that is, for any $j$ and any $(x,y)\in\overline{\Omega_j}$,
$$\Phi\left(\hat\mu\right)+\eps{S}(\alpha{x},\alpha{y})>\Phi\left(\dfrac{k}{c}(z-a_{j}x-b_{j}y)\right).$$
To simplify the notation, we drop the subscript $j$ in $a_j$ and $b_j$. We consider two cases.

Case 1: $\hat\mu\leq\dfrac{k}{c}(z-ax-by)$. Then it is clear that
$$V(x,y,z)>\Phi(\hat\mu)\geq\left(\dfrac{k}{c}(z-ax-by)\right).$$

Case 2: $\hat\mu=\dfrac{z-\alpha^{-1}\varphi(\alpha{x},\alpha{y})}{\sqrt{1+\abs{\nabla\varphi(\alpha{x},\alpha{y})}^2}}>\dfrac{k}{c}(z-ax-by)$. We insert a term $ax+by$ in the left hand side and rearrange the inequality as follows.
$$\dfrac{z-ax-by-\alpha^{-1}(\varphi(\alpha{x},\alpha{y})-a\alpha{x}-b\alpha{y})}{\sqrt{1+\abs{\nabla\varphi(\alpha{x},\alpha{y})}^2}}>\dfrac{k}{c}(z-ax-by),$$
$$\left(\dfrac{c}{\sqrt{1+\abs{\nabla\varphi(\alpha{x},\alpha{y})}^2}}-k\right)(z-ax-by)>\dfrac{c(\varphi(\alpha{x},\alpha{y})-a\alpha{x}-b\alpha{y})}{\alpha\sqrt{1+\abs{\nabla\varphi(\alpha{x},\alpha{y})}^2}},$$
that is,
$$S(\alpha{x},\alpha{y})(z-ax-by)>\dfrac{c(\varphi-h)(\alpha{x},\alpha{y})}{\alpha\sqrt{1+\abs{\nabla\varphi(\alpha{x},\alpha{y})}^2}}.$$
By the definition of $\omega$, we have
$$z-ax-by>\dfrac{c\omega}{\alpha\sqrt{1+\abs{\nabla\varphi(\alpha{x},\alpha{y})}^2}}\geq\dfrac{c\omega}{\alpha}.$$
On the other hand, since $\varphi>h$, we have
$$\hat\mu>\dfrac{z-ax-by}{\sqrt{1+\abs{\nabla\varphi(\alpha{x},\alpha{y})}^2}}.$$
Now
\begin{equation*}\begin{split}
&\quad\,\,V(x,y,z)-\Phi\left(\dfrac{k}{c}(z-ax-by)\right)\\
&\geq\Phi\left(\dfrac{z-ax-by}{\sqrt{1+\abs{\nabla\varphi(\alpha{x},\alpha{y})}^2}}\right)-\Phi\left(\dfrac{k}{c}(z-ax-by)\right)+\eps{S}(\alpha{x},\alpha{y})\\
&=\dfrac{(z-ax-by)S(\alpha{x},\alpha{y})}{c}\\
&\qquad\cdot\int_0^1\!\Phi'\left(\left(\dfrac{t}{\sqrt{1+\abs{\nabla\varphi(\alpha{x},\alpha{y})}^2}}+\dfrac{k}{c}(1-t)\right)(z-ax-by)\right)\,dt\\
&\qquad+\eps{S}(\alpha{x},\alpha{y})\\
\end{split}\end{equation*}
Let us write, for $t\in[0,1]$,
\begin{equation*}\begin{split}
\mu_*=\mu_*(t)
&=\left(\dfrac{t}{\sqrt{1+\abs{\nabla\varphi(\alpha{x},\alpha{y})}^2}}+\dfrac{k}{c}(1-t)\right)(z-ax-by)\\
&=\dfrac{k+tS(\alpha{x},\alpha{y})}{c}(z-ax-by).
\end{split}\end{equation*}
We know from Lemma \ref{S} that $S>0$, thus $\dfrac{z-ax-by}{c}\leq\dfrac{\mu_*}{k}$ and $\mu_*\geq\dfrac{k\omega}{\alpha}$. Now
\begin{equation*}\begin{split}
&\quad\,\,V(x,y,z)-\Phi\left(\dfrac{k}{c}(z-ax-by)\right)\\
&\geq{S}(\alpha{x},\alpha{y})\left(\eps-\dfrac1k\int_0^1\!\mu_*(t)\Phi'(\mu_*(t))\,dt\right)\\
&\geq{S}(\alpha{x},\alpha{y})\left(\eps-\dfrac{1}{k}\sup_{\abs{\mu}\geq\frac{k\omega}{\alpha}}\abs{\mu\Phi'(\mu)}\right)\\
&\geq\dfrac{\eps}{2}S(\alpha{x},\alpha{y})\\
&>0
\end{split}\end{equation*}
This concludes the proof.
\end{proof}

\medskip
\section{Proof of the main result}

Since we have constructed a sub-solution and a super-solution, it suffices to carry out the monotone iteration argument. See \cite{sattinger}, \cite{petrosyan-pop}.


\begin{proof}[Proof of Theorem \ref{main}]
Write $L(u)=\sLap{u}-cu_z+Ku$ where $K>\norm[{L^\infty([-1,1])}]{f'}$ is so large that $\abs{\xi}^{2s}-c\xi_3+K>0$ for any $\xi\in\R^3$.
Using Lemma \ref{solve}, we construct a sequence $\set{w_m}\subset{C}^{2s+\beta}(\R^3)\cap{L}^\infty(\R^3)$ for some $\beta\in(0,1)$ as follows.
Let $w_0=v$. For any $m\geq1$, let $w_m$ be the unique solution of the linear equation
\begin{equation*}
L(w_m)=f(w_{m-1})+Kw_{m-1}.
\end{equation*}

We will prove by induction that
$$-1<v=w_0<w_1<\cdots<w_{m-1}<w_{m}<\cdots<V<1$$
for all $m\geq1$, using the strong maximum principle stated in Lemma \ref{MP}.
If we define $g(u)=f(u)+Ku$, then $g'(u)>0$, showing that $f(u_1)+Ku_1-f(u_2)-Ku_2>0$ if $u_1>{u_2}$.

For $m=1$, we have for each $1\leq{j}\leq{N}$, $L(v_j)=\sLap{v_j}-c(v_j)_z-f(v_j)+f(v_j)+Kv_j=f(v_j)+Kv_j$ and so
$$L(w_1-v_j)=f(v)+Kv-f(v_j)-Kv_j\geq0$$
since $v=\max_{j}v_j\geq{v_j}$ and
$$L(V-w_1)\geq{f}(V)+KV-f(v)-Kv>0.$$
Hence, $v<w_1<V$ unless $w_1=v$, giving a non-smooth solution ${\mathcal{L}}[w_1]=0$, which is impossible by the regularity of $\mathcal{L}$.
By the Schauder estimate \eqref{schauder}, since $v\in{C}^{0,1}(\R^3)\subset{C}^{0,\frac32-s}(\R^3)$, $w_1\in{C}^{2,s-\frac12}(\R^3)$.
In particular, $\sLap{w_1}$ is well-defined by the singular integral.

For any $m\geq2$, if $w_{m-2}<w_{m-1}<V$, then
$$L(w_m-w_{m-1})=f(w_{m-1})+Kw_{m-1}-f(w_{m-2})-Kw_{m-2}\geq0$$
and
$$L(V-w_m)\geq{f}(V)+KV-f(w_{m-1})-Kw_{m-1}>0.$$
Hence, $w_{m-1}<v_m<V$ unless $w_{m-1}=w_m$, in which case ${\mathcal{L}}[w_m]=0$.
Using the Schauder estimate \eqref{schauder} and an interpolation inequality, we have $w_m\in{C}^{2,s-\frac12}(\R^3)$ and
\begin{equation*}\begin{split}
\norm[C^{2,s-\frac12}(\R^3)]{w_m}
&\leq\norm[C^{0,\frac32-s}(\R^3)]{f(w_{m-1})}+K\norm[C^{0,\frac32-s}(\R^3)]{w_{m-1}}\\
&\leq\norm[L^\infty({[-1,1]})]{f}+2K\norm[C^{0,\frac32-s}(\R^3)]{w_{m-1}}\\
&\leq\tilde{C}_8\left(s,K,\norm[L^\infty({[-1,1]})]{f}\right)+\dfrac12\norm[C^{2,s-\frac12}(\R^3)]{w_{m-1}}
\end{split}\end{equation*}
Iterating this yields $\norm[C^{2,s-\frac12}(\R^3)]{w_m}\leq2\tilde{C}_8+\dfrac{1}{2^{m-2}}\norm[C^{2,s-\frac12}(\R^3)]{w_1}$.

Therefore, the sequence $\set{w_m}_{m\geq1}$ is monotone increasing and uniformly bounded in $C^{2,s-\frac12}(\R^3)$.
Hence the pointwise limit
$$u(x)=\lim_{m\to\infty}w_m(x)$$
exists and is unique.
By Arzel\`a-Ascoli theorem, we can extract from each subsequence of $\set{w_m}$ a subsequence converging in $C^{2}(\R^3)$.
Since $u$ is unique, we conclude that
$$w_m\to{u}\qquad\text{in}~C^2(\R^3)$$
Clearly, $\mathcal{L}[u]=0$ in $\R^3$ and the proof is now complete.
\end{proof}

\medskip

\appendix

\section{Properties of the fractional Laplacian}
\label{sect frac lap}
Here we list some elementary properties of the fractional Laplacian mentioned in \cite{silvestre1}, \cite{cabre-sire1}, \cite{palatucci-savin-valdinoci}, \cite{gui-zhao}.

\subsection{Basic properties}

It is convenient to have the following

\begin{lem}\label{C/C}
For any $a\in\R$, we have
$$\int_{\infty}^\infty\!\dfrac{dx}{(x^2+a^2)^{1+s}}=\dfrac{\sqrt{\pi}\Gamma(\frac12+s)}{\Gamma(1+s)}\dfrac{1}{\abs{a}^{1+2s}}=\dfrac{C_{1,s}}{C_{2,s}}\dfrac{1}{\abs{a}^{1+2s}}.$$
As a consequence,
$$\int_{\infty}^\infty\!\dfrac{dx}{(x^2+a^2)^{\frac{n}{2}+s}}=\dfrac{\sqrt{\pi}\Gamma(\frac{n}{2}+s)}{\Gamma(\frac{n+1}{2}+s)}\dfrac{1}{\abs{a}^{n+2s}}=\dfrac{C_{n,s}}{C_{n+1,s}}\dfrac{1}{\abs{a}^{n+2s}}$$
for any integer $n\geq1$.
\end{lem}

\begin{proof}
Without loss of generality, we can assume that $a=1$ by a simple scaling.

Using the substitution $y=1/(x^2+1)$, we have $x=\sqrt{1/y-1}$ and $dx=-(1/2)y^{-3/2}(1-y)^{-1/2}\,dy$, so
\begin{equation*}\begin{split}
\int_{\infty}^\infty\!\dfrac{dx}{(x^2+1)^{1+s}}
&=2\int_{0}^\infty\!\dfrac{dx}{(x^2+1)^{1+s}}\\
&=\int_0^1\!y^{s-\frac12}(1-y)^{-\frac12}\,dy\\
&=\Beta\left(\dfrac12,\dfrac12+s\right)\\
&=\dfrac{\sqrt{\pi}\Gamma\left(\frac12+s\right)}{\Gamma(1+s)},
\end{split}\end{equation*}
where $\Beta$ denotes the Beta function.

The second equality is obtained by replacing $s$ by $s+n/2$.
\end{proof}

\begin{cor}\label{1var}
If $u(x)=u(x_1,\dots,x_n)$ only depends on $x_1$, then
$$\sLap{u}(x_1)=\sLapx{u}(x_1).$$
\end{cor}

\begin{proof}
Clearly, by induction we have
\begin{equation*}\begin{split}
\sLap{u}(x_1)
&=C_{n,s}\PV\int_{\R^n}\!\dfrac{u(x_1)-u(x_1+\xi_1)}{(\xi_1^2+\cdots+\xi_n^2)^{\frac{n}{2}+s}}\,d{\xi_n}\cdots{d}{\xi_1}\\
&=C_{n-1,s}\PV\int_{\R^{n-1}}\!\dfrac{u(x_1)-u(x_1+\xi_1)}{(\xi_1^2+\cdots+\xi_{n-1}^2)^{\frac{n-1}{2}+s}}\,d{\xi_{n-1}}\cdots{d}{\xi_1}\\
&\,\,\,\vdots\\
&=C_{1,s}\PV\int_{\R}\!\dfrac{u(x_1)-u(x_1+\xi_1)}{\abs{\xi_1}^{1+2s}}\,d{\xi_1}\\
&=\sLapx{u}(x_1).
\end{split}\end{equation*}
\end{proof}

\begin{lem}[Homogeneity]\label{homo}
For any admissible $u:\R^n\to\R$, $x\in\R^n$ and $a\in\R$,
$$\sLap(u(ax))=\abs{a}^{2s}\sLap{u}(ax).$$
In particular, if $u$ is an even function then so is $\sLap{u}$.
\end{lem}

\begin{proof}
This is trivial and follows from a change of variable $\eta=a\xi$. Indeed,
\begin{equation*}\begin{split}
\sLap(u(ax))
&=C_{n,s}\PV\int_{\R^n}\!\dfrac{u(ax)-u(ax+a\xi)}{\abs{\xi}^{n+2s}}\,d\xi\\
&=C_{n,s}\abs{a}^{2s}\PV\int_{\R^n}\!\dfrac{u(ax)-u(ax+a\xi)}{\abs{a\xi}^{n+2s}}\abs{a}^{n}\,d\xi\\
&=C_{n,s}\abs{a}^{2s}\PV\int_{\R^n}\!\dfrac{u(ax)-u(ax+\eta)}{\abs{\eta}^{n+2s}}\,d\eta\\
&=\abs{a}^{2s}\sLap{u}(ax).
\end{split}\end{equation*}
\end{proof}

\begin{lem}[Commuting with rigid motions]\label{rigid}
Suppose $M:\R^n\to\R^n$ is a rigid motion, that is, for $x\in\R^n$, $Mx=Ax+b$ where $A\in{O}(n)$ and $b\in\R^n$. Then
$$\sLap(u(Mx))=(\sLap{u})(Mx).$$
\end{lem}

\begin{proof}
It is equivalent to showing that
\begin{equation*}\begin{split}
&C_{n,s}\PV\int_{\R^n}\!\dfrac{u(Ax+b)-u(Ax+b+A\xi)}{\abs{\xi}^{n+2s}}\,d{\xi}\\
&\qquad=C_{n,s}\PV\int_{\R^n}\!\dfrac{u(Ax+b)-u(Ax+b+\xi)}{\abs{\xi}^{n+2s}}\,d{\xi}.
\end{split}\end{equation*}
But since $\abs{A\xi}=\abs{\xi}$ and $\abs{\det{A}}=1$, the result follows from a change of variable $\xi\mapsto{A}\xi$.
\end{proof}

Combining the above simple results, we obtain a useful chain rule.

\begin{lem}[Chain rule for linear transformation]\label{chain}
For any admissible $u:\R^n\to\R$, and $x,a\in\R^n$, there holds
$$\sLap(u(a\cdot{x}))=\abs{a}^s\sLapx{u}(a\cdot{x}).$$
\end{lem}

\begin{proof}
By Lemma \ref{homo}, we may assume that $\abs{a}=1$. By extending $v_1=a\in\R^n$ to an orthonormal basis $\set{v_1,\dots,v_n}$ in $\R^n$, we can construct an orthogonal matrix $A\in{O}(n)$ whose $i$-th row is the row vector $v_i$. In particular, $(Ax)_1=a\cdot{x}$. By Lemma \ref{rigid} and Corollary \ref{1var}, we have
\begin{equation*}\begin{split}
\sLap(u(a\cdot{x}))
&=\sLap(u((Ax)_1))\\
&=(\sLap{u})((Ax)_1)\\
&=(\sLapx{u})((Ax)_1)\\
&=\sLapx{u}(a\cdot{x}).
\end{split}\end{equation*}
\end{proof}

\subsection{Decay properties}

If a function $u$ decays together with its derivatives $\nabla{u}$ and $D^2{u}$ at infinity, $\sLap{u}$ gains a decay of order $2s$, but never better than $O\left(\abs{x}^{-2s}\right)$ because of its nonlocal nature. In a more subtle case when $D^2{u}(x)$ does not decay in $\abs{x}$, we can still get a decay of order $1$ from $\nabla{u}$. The precise statement is as follows.

\begin{lem}[Decay of $\sLap{u}$]\label{decay1}
Suppose $u\in L^{\infty}(\R^n)$ and $u$ is $C_s^{2}$ outside a set $E$ containing the origin. For any $x\in\R^n$ with $\dist(x,E)\gg{R}>r>0$, we have
\begin{equation*}\begin{split}
&\quad\,\,\abs{\sLap{u}(x)}\\
&\leq{C_{n,s}}\Bigg(\dfrac{r^{2-2s}}{4(1-s)}\norm[\dot{C}^2\left(B_{r}(x)\right)]{u}+\\
&\qquad\,\,+\dfrac{1}{2s-1}\left(\dfrac{1}{r^{2s-1}}-\dfrac{1}{R^{2s-1}}\right)\norm[\dot{C}^1\left(B_R\backslash{B_r}(x)\right)]{u}+\dfrac{1}{sR^{2s}}\norm[L^\infty(\R^n)]{u}\Bigg).
\end{split}\end{equation*}
where
$$\norm[\dot{C}^1(\Omega)]{u}=\sum_{j=1}^{n}\sup_{x\in{\Omega}}\abs{\dfrac{\p{u}}{\p{x_j}}(x)},\quad\text{and}\quad\norm[\dot{C}^2(\Omega)]{u}=\sum_{i,j=1}^n\sup_{x\in{\Omega}}\abs{\dfrac{\p^2{u}}{\p{x_i}\p{x_j}}(x)}.$$
\end{lem}

We think of $E$ as the set of all edges of the pyramid projected on $\R^2$. In particular, by taking $r=R=1$, $r=R=\dfrac{\abs{x}}{2}$ and $r=1$, $R=\dfrac{\abs{x}}{2}$ respectively, we have

\begin{cor}\label{decay2}
There exists a constant $C_\Delta=C_\Delta(n,s)$ such that
$$\abs{\sLap{u}(x)}\leq{C_\Delta}\left(\norm[\dot{C}^2\left(B_{1}(x)\right)]{u}+\norm[L^\infty(\R^n)]{u}\right),$$
$$\abs{\sLap{u}(x)}\leq{C_\Delta}\left(\norm[\dot{C}^2\left(B_{\frac{\abs{x}}{2}}(x)\right)]{u}\abs{x}^2+\norm[L^\infty(\R^n)]{u}\right)\abs{x}^{-2s}$$
and
$$\abs{\sLap{u}(x)}\leq{C_\Delta}\left(\norm[\dot{C}^2\left(B_1(x)\right)]{u}+\norm[\dot{C}^1\left(B_{\frac{\abs{x}}{2}}(x)\right)]{u}+\norm[L^\infty(\R^n)]{u}\abs{x}^{-2s}\right).$$
\end{cor}

\begin{proof}[Proof of Lemma \ref{decay1}]
Integrating in $B_r(0)$, $B_R\backslash{B_r}(0)$ and $B_R(0)^c$, we have
$$\sLap{u}(x)=C_{n,s}(I_1+I_2+I_3)$$
where
\begin{equation*}\begin{split}
I_1&=\int_{B_{r}(0)}\!\dfrac{u(x)-u(x+\xi)-\nabla{u}(x)\cdot\xi}{\abs{\xi}^{n+2s}}\,d\xi\\
I_2&=\int_{B_{R}\backslash{B_r}(0)}\!\dfrac{u(x)-u(x+\xi)}{\abs{\xi}^{n+2s}}\,d\xi\\
I_3&=\int_{B_{R}^c}\!\dfrac{u(x)-u(x+\xi)}{\abs{\xi}^{n+2s}}\,d\xi
\end{split}\end{equation*}
Using second order Taylor expansion
$$u(x+\xi)=u(x)+\nabla{u}(x)\cdot\xi+\dfrac12\xi^{T}D^2{u}(\tilde{x})\xi$$
where $\tilde{x}\in B_{r}(x)$, we estimate
\begin{equation*}\begin{split}
\abs{I_1}
&\leq\int_{B_{r}(0)}\!\dfrac{\norm[\infty]{D^2{u}(\tilde{x})}\abs{\xi}^2}{2\abs{\xi}^{n+2s}}\,d\xi\\
&\leq\dfrac12\norm[\dot{C}^2({B_{r}(x)})]{u}\int_{B_{r}(0)}\!\dfrac{d\xi}{\abs{\xi}^{n+2s-2}}\\
&\leq\dfrac12\norm[\dot{C}^2({B_{r}(x)})]{u}\int_0^{r}\!\dfrac{d\rho}{\rho^{2s-1}}\\
&\leq\dfrac{r^{2-2s}}{4(1-s)}\norm[\dot{C}^2({B_{r}(x)})]{u}.
\end{split}\end{equation*}
For the second term, we have
\begin{equation*}\begin{split}
\abs{I_2}
&\leq\int_{B_R\backslash{B_r}(0)}\!\dfrac{1}{\abs{\xi}^{n+2s}}\int_0^1\!\abs{\nabla{u}(x+t\xi)\cdot\xi}\,d{t}d{\xi}\\
&\leq\norm[L^\infty(B_R\backslash{B_r}(x))]{\nabla{u}}\int_{B_R\backslash{B_r}(0)}\!\dfrac{d\xi}{\abs{\xi}^{n-1+2s}}\\
&\leq\norm[\dot{C}^1(B_R(x))]{u}\int_r^R\!\dfrac{d\rho}{\rho^{2s}}\\
&\leq\dfrac{1}{2s-1}\left(\dfrac{1}{r^{2s-1}}-\dfrac{1}{R^{2s-1}}\right)\norm[\dot{C}^1(B_R(x))]{u}
\end{split}\end{equation*}
The last term is simpler and is estimated as
\begin{equation*}\begin{split}
\abs{I_3}
&\leq2\norm[L^\infty(\R^n)]{u}\int_{B_{R}(0)^c}\!\dfrac{d\xi}{\abs{\xi}^{n+2s}}\\
&\leq2\norm[L^\infty(\R^n)]{u}\int_{R}^\infty\!\dfrac{d\rho}{\rho^{1+2s}}\\
&\leq\dfrac{1}{sR^{2s}}\norm[L^\infty(\R^n)]{u}.
\end{split}\end{equation*}
This completes the proof.
\end{proof}

\subsection{The linear operator}

Consider the linear operator $L$ acting on functions $u\in{C_s^{2}}(\R^n)$ by
$$L_0(u)=\sLap{u}+b_0\cdot\nabla{u}+c_0u,$$
where $b_0\in\R^n$ and $c_0\in\R$.

\begin{lem}[Strong maximum principle]\label{MP}
Suppose $u\in{C}_s^2(\R^n)$, $L_0(u)\geq0$, $u(x)\to0$ as $\abs{x}\to\infty$ and $c_0\geq0$.
Then either $u(x)>0$, $\forall x\in\R^n$, or $u(x)\equiv0$, $\forall x\in\R^n$.
\end{lem}

\begin{proof}
Since $u(x)\to0$ as $\abs{x}\to\infty$, we know that $\inf_{\R^n}u$ is attained at some $x_1\in\R^n$. Suppose $u(x)\not\equiv0$. At the global minimum $x_1$, we have $\sLap{u}(x_1)<0$ and $\nabla{u}(x_1)=0$. If $u(x_1)\leq0$, then $c_0u(x_1)\leq0$ and $L_0[u](x_1)\leq\sLap{u}(x_1)<0$, a contradiction.
\end{proof}

\begin{lem}[Solvability of the linear equation]\label{solve}
Assume that $c_0$ is so large that the symbol $\abs{\xi}^{2s}-c\cdot\xi+M>0$ for any $\xi\in\R^n$. Let $\beta\in(0,1)$ be such that $2<2s+\beta<3$.

Then there exists a constant $C_{L_0}=C_{L_0}(n,s,\beta,\abs{b_0},c_0)$ such that for any $f_0\in{C}^\beta(\R^n)$,
there is a unique solution $u\in{C}^{2,2s+\beta-2}(\R^n)$ to the linear equation
$$L_0(u)=f_0,\quad\text{in}~\R^n,$$
satisfying the Schauder estimate
\begin{equation}\label{schauder}
\norm[C^{2,2s+\beta-2}(\R^n)]{u}\leq{C}_{L_0}\norm[C^\beta(\R^n)]{f_0}.
\end{equation}
\end{lem}

The proof involves taking a Fourier transform and a density argument, see \cite{krylov1}, \cite{petrosyan-pop}.

\medskip
\section{Decay of the second order derivative of the profile}
\label{sect 1Ddecay}

\begin{prop}
Let $\Phi$ be as in Theorem \ref{1Dexist}. As $\abs{\mu}\to\infty$, we have
$$\Phi''(\mu)=O\left(\abs{\mu}^{-1-2s}\right).$$
\end{prop}

\begin{cor}\label{Phi''}
For $1/2\leq{s}<1$, there exists a constant $C_\Phi=C_\Phi(s)$ such that
$$\sup_{\mu\in\R}\,\abs{\mu^i\Phi^{(j)}(\mu)}\leq{C}_\Phi$$
for $0\leq{i}\leq{j}$ and $1\leq{j}\leq2$.
Equivalently, if $\Phi_\alpha(\mu)=\Phi(\alpha^{-1}\mu)$, then
$$\sup_{{\mu}\in\R}\,\abs{\mu\Phi_\alpha'(\mu)},\,\sup_{{\mu}\in\R}\,\abs{\mu^2\Phi_\alpha''(\mu)}\leq{C_\Phi},$$
$$\sup_{{\mu}\in\R}\,\abs{\Phi_\alpha'(\mu)},\,\sup_{{\mu}\in\R}\,\abs{\mu\Phi_\alpha''(\mu)}\leq{C_\Phi}\alpha^{-1},$$
and
$$\sup_{{\mu}\in\R}\,\abs{\Phi_\alpha''(\mu)}\leq{C_\Phi}\alpha^{-2}.$$
\end{cor}

The proof relies on a comparison with the almost-explicit layer solution \cite{cabre-sire2}, see also \cite{gui-zhao}. For $t>0$ and $\mu\in\R$, define
$$p_t(\mu)=\dfrac{1}{\pi}\int_0^\infty\!\cos(\mu{r})e^{-tr^{2s}}\,dr$$
and
$$v_t(\mu)=-1+2\int_{-\infty}^{\mu}\!p(t,r)\,dr.$$
They are smooth functions on $\R$ and $v_t$ is a layer solution of
$$\sLapx{v_t}(\mu)=f_t(v_t(\mu)), \quad \forall \mu\in\R$$
where $f_t\in C^2([-1,1])$ is an odd bistable nonlinearity satisfying $f_t(\pm1)=-\dfrac1t$. Moreover, the asymptotic behaviors of $v_t'$ and $v_t''$ are given by
$$\lim_{\abs{\mu}\to\infty}\abs{\mu}^{1+2s}v_t'(\mu)=\dfrac{4ts\Gamma(2s)\sin(\pi{s})}{\pi}>0$$
$$\lim_{\mu\to\pm\infty}\abs{\mu}^{2+2s}v_t''(\mu)=\mp\dfrac{4t^{1-\frac1s}s(1+2s)\Gamma(2s)\sin{\pi{s}}}{\pi}.$$

\begin{proof}[Proof of Proposition \ref{Phi''}]
Differentiating (\ref{1D}) twice, we see that $\Phi''$ satisfies
$$\sLap{\Phi''(\mu)}-k\Phi'''(\mu)=f'(\Phi(\mu))\Phi''(\mu)+f''(\Phi(\mu))(\Phi'(\mu))^2.$$
Let $w_{M,t}(\mu)=Mv_t'(\mu)+\Phi''(\mu)$. Then
\begin{equation*}\begin{split}
&\quad\ \sLap{w_{M,t}}(\mu)-kw_{M,t}'(\mu)+\dfrac4tw_{M,t}(\mu)\\
&=Mv_t'(\mu)\left(\dfrac2t+f_t'(v_t(\mu))\right)+M\left(\dfrac{v_t'(\mu)}{t}-kv_t''(\mu)\right)\\
&\quad+\left(\dfrac{Mv_t'(\mu)}{t}+f''(\Phi(\mu))(\Phi'(\mu))^2\right)+\Phi''(\mu)\left(\dfrac4t+f'(\Phi(\mu))\right)
\end{split}\end{equation*}
Using the facts that $f_t'(\pm1)=-\dfrac1t$, $f'(\pm1)<0$, $\displaystyle\lim_{\mu\to\pm\infty}v_t(\mu)=\pm1$ and $\displaystyle\lim_{\mu\to\pm\infty}\Phi(\mu)=\mp1$, we can find large $T\geq1$ and $\tilde{r}_1\geq1$ such that
$$\dfrac2T+f_T'(v_T(\mu))>0 \quad \text{ and } \quad \dfrac4T+f'(\Phi(\mu))<0.$$
By the positivity of $v_T'(\mu)$, the boundedness of $f''$ and the asymptotic behaviors
$$v_T'(\mu)=O\left(\abs{\mu}^{-1-2s}\right), \quad v_T''(\mu)=O\left(\abs{\mu}^{-2-2s}\right) \text{ and } \quad \Phi'(\mu)=O\left(\abs{\mu}^{-1-2s}\right),$$
there exist large numbers $M_1\geq1$ and $\tilde{r}_2\geq\tilde{r}_1$ such that if $M\geq M_1$, then
$$\dfrac{v_T'(\mu)}{T}-kv_T''(\mu)>0 \quad \text{ and } \quad \dfrac{M_1v_T'(\mu)}{T}+f''(\Phi(\mu))(\Phi'(\mu))^2>0, \quad \forall \abs{\mu}\geq\tilde{r}_2.$$
Therefore, for $M\geq M_1$ and $x\in\set{\abs{\mu}\geq\tilde{r}_2}\cap\set{\Phi''<0}$ there holds
$$\sLapx{w_{M,T}(\mu)}-kw_{M,T}'(\mu)+\dfrac4tw_{M,T}(\mu)>0.$$
Since $v_T'>0$, there exists $M_2\geq M_1$ such that
$$w_{M_2,T}(\mu)=M_2v_T'(\mu)\geq1>0, \quad \forall \abs{\mu}\leq\tilde{r}_2+1.$$

Now we argue by maximum principle that $w_{M_2,T}(\mu)\geq0$ for all $\mu\in\R$. Suppose on the contrary that
$$\inf_{\mu\in\R}w_{M_2,T}(\mu)<0.$$
Since $w_{M_2,T}(\mu)$ decays as $\abs{\mu}\to\infty$, the infimum is attained at some $\tilde{\mu}\in\R$. But then
$$\sLapx{w_{M_2,T}}(\tilde{\mu})<0, \quad w_{M_2,T}'(\tilde{\mu})=0 \quad \text{ and } \quad w_{M_2,T}(\tilde{\mu})<0,$$
which yields $$\sLapx{w_{M_2,T}}(\tilde{\mu})-kw_{M_2,T}'(\tilde{\mu})+\dfrac{4}{T}w_{M_2,T}(\tilde{\mu})<0.$$ On the other hand, we see that $\abs{\tilde{\mu}}>\tilde{r}_2+1$ by the choice of $M_2$ and $\Phi''(\tilde{\mu})<0$ by the definition of $w_{M_2,T}$, giving
$$\sLapx{w_{M_2,T}}(\tilde{\mu})-kw_{M_2,T}'(\tilde{\mu})+\dfrac{4}{T}w_{M_2,T}(\tilde{\mu})>0,$$ a contradiction. Hence,
$$M_2v_T'(\mu)+\Phi''(\mu)\geq0, \quad \forall \mu\in\R.$$

We can now repeat the whole argument, replacing $\Phi''(\mu)$ by $-\Phi''(\mu)$, to obtain
$$M_2v_T'(\mu)-\Phi''(\mu)\geq0, \quad \forall \mu\in\R.$$
Now, for $\mu\neq0$,
$$\abs{\Phi''(\mu)}\leq M_2v_T'(\mu)\leq C\abs{\mu}^{-1-2s}$$
for some constant $C$. This finishes the proof.
\end{proof}

\end{document}